\documentclass[12pt,leqno]{amsart}
\usepackage{amssymb}

\newtheorem{global-theorem}{Theorem}
\newtheorem{theorem}{Theorem}[section]
\newtheorem{lemma}[theorem]{Lemma}
\newtheorem{remark}[theorem]{Remark}

\newtheorem{corollary}[theorem]{Corollary}
\newtheorem{hypothesis}[theorem]{Hypothesis}

\newtheorem{definition}[theorem]{Definition}
\newtheorem{proposition}[theorem]{Proposition}
\newtheorem{prop-def}[theorem]{Proposition-Definition}
\newtheorem{lemma-def}[theorem]{Lemma-Definition}
\newtheorem{question}[theorem]{Question}

\numberwithin{equation}{section}

\begin{document}
\author[A. Aliouche]{Abdelkrim Aliouche}
\address{Laboratoire des syst\`{e}mes dynamiques et contr\^{o}le, 
Universit\'{e} Larbi Ben M'Hidi, Oum-El-Bouaghi, 04000, Alg\'{e}rie}
\email{alioumath@gmail.com}
\author[C. Simpson]{Carlos Simpson}
\address{CNRS, Laboratoire J. A. Dieudonn\'e, UMR 6621 \\
Universit\'e de Nice-Sophia Antipolis\\
06108 Nice, Cedex 2, France}
\email{carlos@unice.fr}
\title[Approximate categories]{Approximate categorical structures}

\begin{abstract}
We consider notions of metrized categories, and then approximate categorical 
structures defined by a function of three variables generalizing the notion 
of $2$-metric space. We prove an embedding theorem giving sufficient conditions 
for an approximate categorical structure to come from an inclusion into a
metrized category.
\end{abstract}

\maketitle

\section{Introduction}

G\"{a}hler \cite{Gahler} introduced the notion of \emph{$2$-metric space}
which is a set $X$ together with a function called the $2$-metric $%
d(x,y,z)\in {\mathbb{R}}$ satisfying some properties generalizing the axioms
for a metric space. Notably, the triangle inequality generalizes to the 
\emph{tetrahedral inequality} for a $2$-metric 
\begin{equation*}
d(x,y,w)\leq d(x,y,z)+d(y,z,w)+d(x,z,w).
\end{equation*}%
One of the main examples of a $2$-metric is obtained by setting $d(x,y,z)$
equal to the area of the triangle spanned by $x,y,z$. Here, we consider
triangles with straight edges. One might imagine considering more generally
triangles with various paths as edges. In this case, in addition to $x$, $y$
and $z$, we should specify a path $f$ from $x$ to $y$, a path $g$ from $y$
to $z$ and a path $h$ from $x$ to $z$. We could then set $d(f,g,h)$ to be
the area of the figure spanned by these paths, more precisely the minimal
area of a disk whose boundary consists of the circle formed by these three
paths.

This generalization takes us in the direction of category theory: we may
think of $d(f,g,h)$ as being some kind of distance between $h$ and a
\textquotedblleft composition\textquotedblright\ of $f$ and $g$. We will
formalize this notion here and call it an \emph{approximate categorical
structure}.

Generalizing the notion of $2$-metric space in this direction may be viewed
as directly analogue to the recent paper of Weiss \cite{Weiss} in which he
proposed the notion of \textquotedblleft metric $1$-space\textquotedblright\
which was a category together with a \textquotedblleft distance
function\textquotedblright\ $d(f)$ for arrows $f:x\rightarrow y$, which
would then be required to satisfy the analogues of the usual axioms of a
metric space. In his setup, the pair $(x,y)$ is replaced by a pair of objects plus an
arrow $f$ from $x$ to $y$.

In our situation, we would like to generalize the notion of $2$-metric in a
similar way replacing a triple of points $(x,y,z)$ by a triple of objects
together with arrows $f:x\rightarrow y$, $f:y\rightarrow z$ and $%
h:x\rightarrow z$. In this setup, we don't need to start with a category but
only with a graph and the $2$-metric itself represents some kind of
approximation of the notion of composition.

In an approximate categorical structure, then the underlying set-theoretical
object is a graph, consisting of a set of objects $X$ and sets of arrows $%
A(x,y)$ for any $x,y\in X$. The distance function $d(f,g,h)$ is required to
be defined whenever $f\in A(x,y)$, $g\in A(y,z)$ and $h\in A(x,z)$. The main
axioms, generalizing the tetrahedral axiom of a $2$-metric space, are the 
\emph{left and right associativity properties}. These concern the situation
of a sequence of objects $x,y,z,w$ and arrows going in the increasing
direction: 
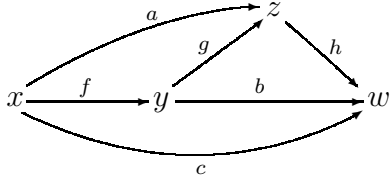
\begin{equation*}
\setlength{\unitlength}{.3mm}%
\begin{picture}(200,120)


\put(20,50){$x$}

\put(85,50){$y$}

\put(135,90){$z$}

\put(180,50){$w$}

\qbezier(29,52)(52,52)(80,52)
\put(78,52){\vector(1,0){5}}
\put(52,56){${\scriptstyle f}$}

\qbezier(94,59)(113,73)(132,87)
\put(129,84){\vector(1,1){5}}
\put(105,75){${\scriptstyle g}$}

\qbezier(144,87)(160,73)(175,59)
\put(171,63){\vector(1,-1){5}}
\put(163,73){${\scriptstyle h}$}

\qbezier(29,59)(80,90)(127,94)
\put(127,94){\vector(1,0){5}}
\put(82,88){${\scriptstyle a}$}

\qbezier(95,52)(144,52)(173,52)
\put(173,52){\vector(1,0){5}}
\put(130,56){${\scriptstyle b}$}

\qbezier(27,47)(104,10)(175,46)
\put(173,45){\vector(3,2){5}}
\put(104,20){${\scriptstyle c}$}



\end{picture}
\end{equation*}%
The left associativity condition says 
\begin{equation*}
d(a,h,c)\leq d(f,g,a)+d(g,h,b)+d(f,b,c).
\end{equation*}%
It means that if $a$ is close to a composition of $f$ and $g$, if $b$ is
close to a composition of $g$ and $h$ and if $c$ is close to a composition
of $f$ and $b$, then $c$ is also close to a composition of $a$ and $h$.

Looking at the same picture but viewed with the arrow $c$ passing along the
top: 
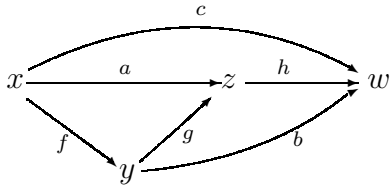
\begin{equation*}
\setlength{\unitlength}{.3mm} 
\begin{picture}(200,80)


\put(20,50){$x$}

\put(70,10){$y$}

\put(115,50){$z$}

\put(180,50){$w$}

\qbezier(29,45)(48,31)(67,17)
\put(64,20){\vector(1,-1){5}}
\put(42,23){${\scriptstyle f}$}

\qbezier(79,17)(95,31)(110,45)
\put(106,41){\vector(1,1){5}}
\put(98,27){${\scriptstyle g}$}

\qbezier(126,52)(149,52)(173,52)
\put(171,52){\vector(1,0){5}}
\put(140,56){${\scriptstyle h}$}

\qbezier(29,52)(70,52)(111,52)
\put(110,52){\vector(1,0){5}}
\put(70,56){${\scriptstyle a}$}

\qbezier(80,13)(135,18)(171,47)
\put(171,47){\vector(2,1){5}}
\put(147,24){${\scriptstyle b}$}

\qbezier(30,58)(104,95)(172,59)
\put(171,60){\vector(3,-2){5}}
\put(104,82){${\scriptstyle c}$}

\end{picture}
\end{equation*}
the right associativity condition says 
\begin{equation*}
d(f, b, c) \leq d(f,g,a) + d(g,h,b) + d(a,h,c). 
\end{equation*}

It is natural to add the data of identity elements $1_{x}\in A(x,x)$ such
that 
\begin{equation*}
d(1_{x},f,f)=0\text{ \ and \ }d(f,1_{y},f)=0\text{.}
\end{equation*}%
Therefore, The theory works out pretty nicely. For example, we obtain a
distance function on the arrow sets 
\begin{equation*}
\mathrm{dist}_{A(x,y)}(f,g):=d(1_{x},f,g).
\end{equation*}%
This is a pseudometric, that is to say it satisfies the triangle inequality
but there might be distinct pairs $f,g$ at distance zero apart. We may
however identify them together. This is discussed in Section \ref{metarrowsets}.

Perhaps a more direct way to introduce a categorical notion such that the
arrow sets are metric spaces, would be just to consider a category enriched
in metric spaces. Here, it will be useful for our development to consider
the enrichment as being with respect to the product structure where the
metric on the product of two metric spaces is the sum of the metrics on the
pieces:%
\begin{equation*}
d((x,x^{\prime }),(y,y^{\prime })):=d(x,y)+d(x^{\prime },y^{\prime })\text{.}
\end{equation*}
We describe this theory first, in Section \ref{metrizedcategories}.

A metrized category then yields an approximate categorical structure, with
the tetrahedral inequalities stated in Proposition \ref{mcac}.

Approximate categorical structures are weaker objects, in that any subgraph
of an AC structure will have an induced AC structure. In particular, if we start with a
metrized category and take any subgraph then we get an AC structure. It is
natural to ask whether an arbitrary AC structure arises in this way. There
is a good notion of \emph{functor} $(X,A,d)\rightarrow (Y,B,d)$ between
two AC structures, see Section \ref{functors}, 
so we can look at functors from an AC structure to
metrized categories. Any such functor induces a distance on the free
category $\mathbf{Free}(X,A)$ generated by the graph $(X,A)$ and we obtain a
distance denoted $d^{\mathrm{max}}$ on $\mathbf{Free}(X,A)$ as the supremum
of these distances. This is discussed in Section \ref{fmc}. The upper bound 
\begin{equation*}
d^{\mathrm{max}}(f,g,h)\leq d(f,g,h)
\end{equation*}
is tautological.

Our main result will be to give a strong lower bound under a certain
hypothesis. Recall that in \cite{AS1} and \cite{AS2}, it was useful to
introduce a new axiom, called \emph{transitivity}, for $2$-metric spaces.
This was a metric version of the idea that given four points, if two triples
are colinear then all four are colinear, especially if the two middle points
aren't too close together.

In Section \ref{sec-trans}, we introduce the analogue of the transitivity axiom for AC
structures in Definition \ref{abstrans}. This axiom turns out to be what is
required in order to be able to define the \emph{Yoneda functors} $Y_{u}$,
for $u\in X$. We would like to set $Y_{u}(x):=A(u,x)$ together with its
distance. This is a metric space and the distance $d(a,f,b)$ allows us to
define a \emph{metric correspondence} from $Y_{u}(x)$ to $Y_{u}(y)$, see
Section \ref{yoneda}. There is a metrized category of (bounded) metric spaces with
morphisms the metric correspondences. If $(X,A,d)$ is transitive, then $Y_{u}
$ is a functor from $(X,A,d)$ to this metrized category.

Existence of these functors yields lower bounds on $d^{\mathrm{max}}(f,g,h)$
and somewhat surprisingly the lower bounds are sharp: we have that
\begin{equation*}
d^{\mathrm{max}}(f,g,h)=d(f,g,h)
\end{equation*}
\linebreak\ whenever $(X,A,d)$ is transitive (also needed are boundedness
and a very weak graph transitivity hypothesis \ref{graphcomp}). We obtain
the following embedding theorem saying that an AC structure with these
properties is obtained as a subgraph of a metrized category.

\begin{theorem}
\label{mainintro} Suppose $(X,A,d)$ is an approximate categorical structure
that is bounded, satisfies the separation property 
(Definition \ref{separated}), is absolutely transitive (Definition \ref{abstrans}) and
satisfies Hypothesis \ref{graphcomp}. Then there exists a metrized category $%
{{\mathcal{C}}}$ with $\mathrm{Ob}({{\mathcal{C}}})=X$ and inclusions $%
A(x,y)\subset {{\mathcal{C}}}(x,y)$ such that for any $f\in A(x,y)$, $g\in
A(y,z)$ and $h\in A(x,z)$, 
\begin{equation*}
d(f,g,h)=d_{{{\mathcal{C}}}}(g\circ _{{{\mathcal{C}}}}f,h).
\end{equation*}
\end{theorem}

In Section \ref{pathsRn} we discuss how the $d^{\rm max}$ construction, applied to 
the standard $2$-metric space with triangle area, gives rise to the category of 
piecewise-linear paths. 
Then, in Section \ref{questions}, we discuss numerous further questions and directions. 

\section{Metrized categories}
\label{metrizedcategories}

A category is a triple $(X,A,\circ )$ where $X$ is the set of objects, $%
A(x,y)$ is the set of arrows from $x$ to $y$ for each pair of objects and $%
g\circ f\in A(x,z)$ whenever $f\in A(x,y)$ and $g\in A(y,z)$. These are
subject to the existence of an identity arrow $1_{x}\in A(x,x)$ satisfying $%
f\circ 1_{x}=f$ and $1_{y}\circ f=f$ for all $f\in A(x,y)$ and the
associativity axiom for $f\in A(x,y)$, $g\in A(y,z)$ and $h\in A(z,w)$
requiring 
\begin{equation*}
h\circ (f\circ g)=(h\circ f)\circ g.
\end{equation*}
We can introduce the notion of \emph{pseudometric structure} on a category
as above, which is by definition a pseudometric $\phi (f,g)$ defined for $%
f,g\in A(x,y)$, satisfying the properties of a pseudometric: 
\begin{equation*}
\phi (f,f)=0,\;\;\phi (f,g)=\phi (g,f)\text{ \ and \ }\phi (f,g)\leq \phi
(f,h)+\phi (h,g)\text{.}
\end{equation*}%
If in addition%
\begin{equation*}
\phi (f,g)=0\Rightarrow f=g\text{,}
\end{equation*}
then it is a metric. This separation condition will be imposed as
appropriate, see also Definition \ref{separated} below.

We require the following compatibility with the structure of category: for
any triple of objects $x,y,z\in X$, the composition function 
\begin{equation*}
A(x,y)\times A(y,z)\rightarrow A(x,z),\;\;\;(f,g)\mapsto g\circ f
\end{equation*}%
should be nonincreasing, where we provide the product on the left with the
metric 
\begin{equation*}
(\phi +\phi )((f,g),(f^{\prime },g^{\prime })):=\phi (f,f^{\prime })+\phi
(g,g^{\prime }).
\end{equation*}%
In concrete terms this is equivalent to requiring that 
\begin{equation}
\phi (g\circ f,g^{\prime }\circ f^{\prime })\leq \phi (f,f^{\prime })+\phi
(g,g^{\prime }).  \label{mcaxiom}
\end{equation}%
We call such a structure a \emph{pseudo-metrized category}. If the separation
property holds we call it a \emph{metrized category}.

If there is no confusion, we denote $\phi (f,f^{\prime })$ by just $\mathrm{%
dist}_{A(x,y)}(f,f^{\prime })$ or $d_{A(x,y)}(f,f^{\prime })$.

As motivation for the next section, if $(X,A,\circ ,\phi )$ is a
pseudo-metrized category, we can define a function of three variables $%
d(f,g,h)$ defined whenever $f\in A(x,y)$, $g\in A(y,z)$ and $h\in A(x,z)$ by
putting 
\begin{equation}
d(f,g,h):=\phi (g\circ f,h).  \label{mc2ac}
\end{equation}

\begin{proposition}
\label{mcac} This function satisfies the following properties: \newline
---for all $f,g\in A(x,y)$ we have%
\begin{equation*}
d(f,1_{y},g)=d(1_{x},f,g)=\phi (f,g)\text{;}
\end{equation*}
\newline
---for all $(f,g,h;a,b;c)$ with 
\begin{equation*}
x\overset{f}{\rightarrow }y\overset{g}{\rightarrow }z\overset{h}{\rightarrow 
}w
\end{equation*}%
and 
\begin{equation*}
x\overset{a}{\rightarrow }z,\;\;y\overset{b}{\rightarrow }w,\;\;x\overset{c}{%
\rightarrow }w
\end{equation*}%
we have: 
\begin{equation*}
d(f,b,c)\leq d(f,g,a)+d(g,h,b)+d(a,h,c)
\end{equation*}%
and 
\begin{equation*}
d(a,h,c)\leq d(f,g,a)+d(g,h,b)+d(f,b,c).
\end{equation*}
\end{proposition}

\begin{proof}
For the first statements note that%
\begin{equation*}
d(f,1_{y},g):=\phi (1_{y}\circ f,g)=\phi (f,g)
\end{equation*}
and 
\begin{equation*}
d(1_{x},f,g):=\phi (f\circ 1_{x},g)=\phi (f,g)\text{.}
\end{equation*}

In the second part, applying the definitions, the first inequality that we
would like to show is equivalent to 
\begin{equation*}
\phi (b\circ f,c)\leq \phi (g\circ f,a)+\phi (h\circ g,b)+\phi (h\circ a,c).
\end{equation*}%
By the triangle inequality in $A(x,w)$ applied to the sequence $b\circ
f,h\circ g\circ f,h\circ a,c$ we have 
\begin{equation*}
\phi (b\circ f,c)\leq \phi (b\circ f,h\circ g\circ f)+\phi (h\circ g\circ
f,h\circ a)+\phi (h\circ a,c).
\end{equation*}%
The composition axiom \eqref{mcaxiom} implies that 
\begin{equation*}
\phi (h\circ g\circ f,h\circ a)\leq \phi (h,h)+\phi (g\circ f,a)=\phi
(g\circ f,a).
\end{equation*}%
Similarly%
\begin{equation*}
\phi (b\circ f,h\circ g\circ f)\leq \phi (b,h\circ g)=\phi (h\circ g,b)\text{%
.}
\end{equation*}%
Therefore we get 
\begin{equation*}
\phi (b\circ f,c)\leq \phi (h\circ g,b)+\phi (g\circ f,a)+\phi (h\circ a,c).
\end{equation*}%
Putting in the definition $d(f,g,h):=\phi (g\circ f,h)$ this gives 
\begin{equation*}
d(f,b,c)\leq d(g,h,b)+d(f,g,a)+d(a,h,c),
\end{equation*}%
which is the first inequality. For the second inequality, we would like to
show 
\begin{equation*}
\phi (h\circ a,c)\leq \phi (g\circ f,a)+\phi (h\circ g,b)+\phi (b\circ f,c).
\end{equation*}%
Apply the triangle inequality in $A(x,w)$ to the sequence $h\circ a,h\circ
g\circ f,b\circ f,c$. We get 
\begin{equation*}
\phi (h\circ a,c)\leq \phi (h\circ a,h\circ g\circ f)+\phi (h\circ g\circ
f,b\circ f)+\phi (b\circ f,c).
\end{equation*}%
As before, \eqref{mcaxiom} implies that 
\begin{equation*}
\phi (h\circ a,h\circ g\circ f)\leq \phi (g\circ f,a)\text{ and }\phi
(h\circ g\circ f,b\circ f)\leq \phi (h\circ g,b).
\end{equation*}
Hence we obtain 
\begin{equation*}
\phi (h\circ a,c)\leq \phi (g\circ f,a)+\phi (h\circ g,b)+\phi (b\circ f,c),
\end{equation*}%
in other words 
\begin{equation*}
d(a,h,c)\leq d(f,g,a)+d(g,h,b)+d(f,b,c)
\end{equation*}%
which is the second inequality.
\end{proof}

Suppose $(X,A,\circ ,\phi )$ is a category with a pseudometric $\phi $.
Define a new set $\tilde{A}(x,y)$ to be the quotient of $A(x,y)$ by the
relation that $x\sim x^{\prime }$ if $\phi (x,x^{\prime })=0$. This is an
equivalence relation. It is compatible with the composition operation by the
axiom \eqref{mcaxiom}. Therefore $\circ $ induces a composition which we
again denote $\circ $ on $(X,\tilde{A})$. Also the distance $\phi $ induces
a metric $\tilde{\phi}$ on $\tilde{A}(x,y)$, and $(X,\tilde{A},\circ ,\tilde{%
\phi})$ is a metrized category satisfying the separation property.

\section{Approximate categories}

\label{approximatecategories}

Abstracting the properties given by Proposition \ref{mcac}, 
we can forget about the composition operation and just
look at the function of three variables 
$d(f,g,h)$.

Consider a set of objects $X$ and for each $x,y\in X$ a set of arrows $A(x,y)
$. Suppose we have isolated an \emph{identity arrow} $1_x\in A(x,x)$ for
each $x\in X$. Consider a \emph{triangular distance function} $d(f,g,h)$
defined whenever 
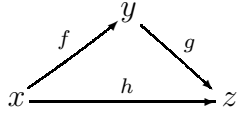
\begin{equation*}
\setlength{\unitlength}{.3mm} 
\begin{picture}(160,90)


\put(20,30){$x$}

\put(70,70){$y$}

\put(115,30){$z$}

\qbezier(29,38)(48,51)(67,66)
\put(64,63){\vector(1,1){5}}
\put(42,57){${\scriptstyle f}$}

\qbezier(79,66)(95,52)(110,38)
\put(106,42){\vector(1,-1){5}}
\put(98,56){${\scriptstyle g}$}

\qbezier(30,32)(70,32)(107,32)
\put(107,32){\vector(1,0){5}}
\put(70,36){${\scriptstyle h}$}

\end{picture}
\end{equation*}
that is to say $f\in A(x,y)$, $g\in A(y,z)$ and $h\in A(x,z)$.

Assume the following axioms:

\noindent \textbf{Identity axioms}---

\noindent Left identity : for all $f\in A(x,y)$ we have%
\begin{equation*}
d(f,1_{y},f)=0\text{;}
\end{equation*}

\noindent Right identity : for all $f\in A(x,y)$ we get%
\begin{equation*}
d(1_{x},f,f)=0\text{;}
\end{equation*}

\noindent \textbf{Associativity axioms}---given a ``tetrahedron'' $(f,g,h;
a,b; c)$ with 
\begin{equation*}
x\overset{f}{\rightarrow} y \overset{g}{\rightarrow} z \overset{h}{%
\rightarrow} w 
\end{equation*}
and 
\begin{equation*}
x \overset{a}{\rightarrow} z, \;\; y \overset{b}{\rightarrow} w, \;\; x 
\overset{c}{\rightarrow} w, 
\end{equation*}

\noindent Left associativity : 
\begin{equation*}
d(a, h, c) \leq d(f,g,a) + d(g,h,b) + d(f,b,c); 
\end{equation*}
Right associativity: 
\begin{equation*}
d(f, b, c) \leq d(f,g,a) + d(g,h,b) + d(a,h,c). 
\end{equation*}

\begin{definition}
\label{acs} An \emph{approximate categorical structure} (\emph{AC-structure}%
) is a triple $(X,A,d)$, together with the specified identities $1_x$,
satisfying the above axioms.
\end{definition}

An \emph{approximate semi-categorical structure} is a triple $(X,A,d)$,
without specified identities, satisfying just the associativity axioms.

\begin{lemma}
Suppose $(X,A,d)$ is an approximate categorical structure. Then for any $%
x,y,z\in X$ with $f\in A(x,y)$, $g\in A(y,z)$ and $h\in A(x,z)$, we have $%
d(f,g,h)\geq 0$.
\end{lemma}

\begin{proof}
For notational simplicity we denote the third map by $a$. Then, Use left
associativity for $(f,g,1_{z};a,g;a)$. It says 
\begin{equation*}
d(a,1_{z},a)\leq d(f,g,a)+d(g,1_{z},g)+d(f,g,a).
\end{equation*}%
Since $d(a,1_{z},a)=0$ and $d(g,1_{z},g)=0$ we get $2d(f,g,a)\geq 0$,
therefore $d(f,g,a)\geq 0$ as claimed.
\end{proof}

\begin{lemma}
Suppose $(X,A,d)$ is an approximate semi-categorical \linebreak\ structure
(resp. categorical structure) and suppose we are given subsets $%
B(x,y)\subset A(x,y)$ (resp. subsets containing $1_{x}$ if $x=y$). Then $%
(X,B,d|_{B})$ is an approximate semi-categorical (resp. categorical)
structure.
\end{lemma}

\begin{proof}
The conditions for $d|_B$ follow from the same conditions for $d$ on $A$.
\end{proof}

\begin{corollary}
Let $(X,C,\circ ,\phi )$ be a pseudo-metrized (resp. metrized) category and
suppose $A(x,y)\subset C(x,y)$ are subsets. Then $(X,A,d|_{A})$ is an
approximate semi categorical structure and if $1_{x}\in A(x,x)$ then we get
an approximate categorical structure (resp. separated AC structure).
\end{corollary}

An approximate categorical structure is clearly also an approximate
semi-categorical structure.

\begin{question}
Suppose $(X,A,d)$ is an approximate semi-categorical structure. Let $%
A^+(x,y):= A(x,y)$ for $x\neq y$ and $A^+(x,y):= A(x,x)\sqcup \{ 1_x\}$. Is
there a natural way to extend $d$ to $A^+$ to obtain an AC structure?
\end{question}

\section{Metric on the arrow sets}

\label{metarrowsets}

\begin{lemma}
If $x,y\in X$ and $f,g\in A(x,y)$ then 
\begin{equation*}
d(f, 1_y,g) = d(1_x, f ,g). 
\end{equation*}
\end{lemma}

\begin{proof}
Use the left associativity axiom for $(1_{x},f,1_{y};f,f;g)$ which says 
\begin{equation*}
d(f,1_{y},g)\leq d(1_{x},f,f)+d(f,1_{y},f)+d(1_{x},f,g)=d(1_{x},f,g).
\end{equation*}%
On the other hand, the right associativity axiom for $(1_{x},f,1_{y};f,f;g)$
gives 
\begin{equation*}
d(1_{x},f,g)\leq d(1_{x},f,f)+d(f,1_{y},f)+d(f,1_{y},g)=d(f,1_{y},g).
\end{equation*}
\end{proof}

By the preceding lemma, we can define a distance on $A(x,y)$ as follows, for 
$f,g\in A(x,y)$ put 
\begin{equation*}
\phi (f,g):=d(1_{x},f,g).
\end{equation*}

Note that for any $x$ we have 
\begin{equation*}
d(1_{x},1_{x},1_{x})=0\text{.}
\end{equation*}

\begin{lemma}
This distance is reflexive: 
\begin{equation*}
\phi (f,f)= 0, 
\end{equation*}
symmetric: 
\begin{equation*}
\phi (f,g)= \phi (g,f), 
\end{equation*}
and satisfies the triangle inequality: 
\begin{equation*}
\phi (f,g) \leq \phi (f,h)+\phi (h,g). 
\end{equation*}
\end{lemma}

\begin{proof}
By definition 
\begin{equation*}
\phi (f,f)=d(f,1_{y},f)=0
\end{equation*}
by the left identity axiom. Using left associativity we have 
\begin{eqnarray*}
\phi (f,g) &=&d(f,1_{y},g) \\
&\leq &d(g,1_{y},f)+d(1_{y},1_{y},1_{y})+d(g,1_{y},g)\text{,}
\end{eqnarray*}%
so 
\begin{equation*}
\phi (f,g)\leq d(g,1_{y},f)=\phi (g,f),
\end{equation*}%
which by symmetry gives $\phi (f,g)=\phi (g,f)$. For the triangle inequality
suppose $f,g,h\in A(x,y)$, then applying left associativity we get 
\begin{eqnarray*}
\phi (f,g) &=&d(f,1_{y},g) \\
&\leq &d(h,1_{y},f)+d(1_{y},1_{y},1_{y})+d(h,1_{y},g),
\end{eqnarray*}%
therefore 
\begin{equation*}
\phi (f,g)\leq \phi (h,f)+\phi (h,g).
\end{equation*}
\end{proof}

\begin{lemma}
\label{phileft} Given $f,f^{\prime }\in A(x,y)$, $h\in A(y,z)$ and $c\in
A(x,z)$ we have 
\begin{equation*}
d(f,h,c)\leq d(f^{\prime },h,c) + \phi (f,f^{\prime }) . 
\end{equation*}
\end{lemma}

\begin{proof}
Applying right associativity with $(f,1_{y},h;f^{\prime },h;c)$, that is $%
g:=1_{y}$, $a:=f^{\prime }$ and $b:=h$ we get 
\begin{equation*}
d(f,b,c)\leq d(f,g,a)+d(g,h,b)+d(a,h,c),
\end{equation*}%
which in our case says 
\begin{equation*}
d(f,h,c)\leq d(f,1_{y},f^{\prime })+d(1_{y},h,h)+d(f^{\prime },h,c).
\end{equation*}%
Since $\phi (f,f^{\prime })=d(f,1_{y},f^{\prime })$ and $d(1_{y},h,h)=0$ we
obtain the desired statement.
\end{proof}

Similarly,

\begin{lemma}
\label{phiright} Given $f\in A(x,y)$, $h,h^{\prime }\in A(y,z)$ and $c\in
A(x,z)$ we have 
\begin{equation*}
d(f,h,c)\leq d(f,h^{\prime },c) + \phi (h,h^{\prime }) . 
\end{equation*}
\end{lemma}

\begin{proof}
Applying left associativity with $(f,1_{y},h;f,h^{\prime };c)$, that is $%
g:=1_{y}$, $a:=f$ and $b:=h^{\prime }$ we obtain 
\begin{equation*}
d(a,h,c)\leq d(f,g,a)+d(g,h,b)+d(f,b,c),
\end{equation*}%
which in our case says 
\begin{equation*}
d(f,h,c)\leq d(f,1_{y},f)+d(1_{y},h,h^{\prime })+d(f,h^{\prime },c).
\end{equation*}%
As 
\begin{equation*}
d(f,1_{y},f)=0\text{ and }d(1_{y},h,h^{\prime })=\phi (h,h^{\prime })\text{,}
\end{equation*}%
we get the desired statement.
\end{proof}

For the third edge,

\begin{lemma}
\label{phicomp} Given $f\in A(x,y)$, $g\in A(y,z)$ and $c\in A(x,z)$ we have 
\begin{equation*}
d(f,g,c)\leq d(f,g,c^{\prime }) + \phi (c,c^{\prime }) . 
\end{equation*}
\end{lemma}

\begin{proof}
Applying right associativity with $(f,g,1_{z};c^{\prime },g;c)$, that is $%
h:=1_{z}$, $a:=c^{\prime }$ and $b:=g$ we have 
\begin{equation*}
d(f,b,c)\leq d(f,g,a)+d(g,h,b)+d(a,h,c),
\end{equation*}%
which in our case says 
\begin{equation*}
d(f,g,c)\leq d(f,g,c^{\prime })+d(g,1_{z},g)+d(c^{\prime },1_{z},c).
\end{equation*}%
Since%
\begin{equation*}
\phi (c,c^{\prime })=d(c^{\prime },1_{z},c)\text{ and }d(g,1_{z},g)=0\text{,}
\end{equation*}
we obtain the desired statement.
\end{proof}

Putting these together we get:

\begin{corollary}
\label{dphicorollary} Given $f,f^{\prime }\in A(x,y)$, $g,g^{\prime }\in
A(y,z)$ and $h,h^{\prime }\in A(x,z)$ we have 
\begin{equation*}
d(f,g,h)\leq d(f^{\prime },g^{\prime },h^{\prime }) + \phi (f,f^{\prime }) +
\phi (g,g^{\prime }) + \phi (h,h^{\prime }) 
\end{equation*}
\end{corollary}

\begin{proof}
By the above lemmas we get successively 
\begin{eqnarray*}
d(f,g,h) &\leq &d(f^{\prime },g,h)+\phi (f,f^{\prime }) \\
&\leq &d(f^{\prime },g^{\prime },h)+\phi (f,f^{\prime })+\phi (g,g^{\prime })
\end{eqnarray*}%
\begin{equation*}
\leq d(f^{\prime },g^{\prime },h^{\prime })+\phi (f,f^{\prime })+\phi
(g,g^{\prime })+\phi (h,h^{\prime }).
\end{equation*}
\end{proof}

\begin{definition}
\label{separated} We say that an AC structure is \emph{separated} if%
\begin{equation*}
\phi (f,f^{\prime })=0\Rightarrow f=f^{\prime }\text{.}
\end{equation*}
Equivalently, each $(A(x,y),\phi )$ is a metric space rather than a
pseudometric space.
\end{definition}

The separation property may be ensured by a quotient construction. Given an
AC structure in general, define the relation that%
\begin{equation*}
f\sim f^{\prime }\text{ if \ }\phi (f,f^{\prime })=0\text{.}
\end{equation*}

\begin{lemma}
This is an equivalence relation on $A(x,y)$. Let $\tilde{A}(x,y):=
A(x,y)/\sim$. The distance function $d(f,g,h)$ passes to the quotient to be
a function of $f\in \tilde{A}(x,y)$, $g\in \tilde{A}(y,z)$ and $h\in A(x,z)$%
. Then $(X,\tilde{A}, d)$ is a separated AC structure.
\end{lemma}

\begin{proof}
It is an equivalence relation by the triangle inequality of $\phi$. The
above corollary says that $d$ passes to the quotient. The axioms hold to get
an AC structure.
\end{proof}

The lemma shows that an approximate categorical structure can always be
replaced by one which satisfies the separation property. We will generally
assume that this has been done.

\begin{lemma}
Suppose $(X,A,d)$ is a separated AC structure. The function $d$ is
continuous on the topologies associated to the metric spaces $A(\cdot ,
\cdot )$. More precisely, for any $x,y,z\in X$, 
\begin{equation*}
d: A(x,y)\times A(y,z)\times A(x,y)\rightarrow {\mathbb{R}} 
\end{equation*}
is a continuous function of its three variables.
\end{lemma}

\begin{proof}
This also follows from Corollary \ref{dphicorollary}.
\end{proof}

\section{The $\protect\epsilon$-categoric condition}

We say that $(X,A,d)$ is \emph{$\epsilon $-categoric} if for any $f\in A(x,y)
$ and $g\in A(y,z)$ there exists $h\in A(x,z)$ such that 
\begin{equation*}
d(f,g,h)\leq \epsilon \text{.}
\end{equation*}

\begin{lemma}
Given $f\in A(x,y)$, $g\in A(y,z)$ and $a,a^{\prime }\in A(x,z)$ we have 
\begin{equation*}
\phi (a,a^{\prime }) \leq d(f,g,a) + d(f,g,a^{\prime }). 
\end{equation*}
\end{lemma}

\begin{proof}
Applying left associativity with $(f,g,1_{z};a,g;a^{\prime })$, that is  $%
h:=1_{z}$, $b:=g$ and $c:=a^{\prime }$ we get 
\begin{equation*}
d(a,h,c)\leq d(f,g,a)+d(g,h,b)+d(f,b,c),
\end{equation*}%
which in our case says 
\begin{equation*}
d(a,1_{z},a^{\prime })\leq d(f,g,a)+d(g,1_{z},g)+d(f,g,a^{\prime }).
\end{equation*}%
As $d(g,1_{z},g)=0$ and $d(a,1_{z},a^{\prime })=\phi (a,a^{\prime })$ we
obtain the desired statement.
\end{proof}

\begin{corollary}
\label{phizero} If 
\begin{equation*}
d(f,g,h)=d(f,g,h^{\prime })=0\text{,}
\end{equation*}
then $\phi (h,h^{\prime })=0$. In particular, if the separation property
(Definition \ref{separated}) is satisfied then it implies that $h=h^{\prime }
$.
\end{corollary}

\section{Amplitude}

We consider now a new function for any $x,y\in X$ denoted 
\begin{equation*}
\alpha :A(x,y)\rightarrow {\mathbb{R}}.
\end{equation*}%
We would like to consider $\alpha $ as analogous to the metric function
considered by Weiss \cite{Weiss}, that is $\alpha (f)$ is supposed to
represent the \textquotedblleft length\textquotedblright\ of $f$.

We ask first that $\alpha$ satisfy the \emph{reflexivity axiom} $\alpha
(1_x)=0$.

Recall that if $(X,d)$ is a $2$-metric space then the distance function $%
\varphi (x,y)$ satisfied 
\begin{equation*}
\varphi (x,z)\leq \varphi (x,y)+\varphi (y,z)+d(x,y,z).
\end{equation*}%
With this motivation, we ask that $\alpha $ satisfy the \emph{triangle
inequality}: for any $x,y,z\in X$ and $f\in A(x,y)$, $g\in A(y,z)$ and $h\in
A(x,z)$, we should have 
\begin{equation*}
\alpha (h)\leq \alpha (f)+\alpha (g)+d(f,g,h).
\end{equation*}%
Recall furthermore that Weiss imposed an additional axiom for his metric
function, namely (in his notations \cite{Weiss} that 
\begin{equation*}
|\varphi (u)-\varphi (v)|\leq \varphi (v\circ u).
\end{equation*}%
Transposed into our approximately categorical situation, we therefore add
the \emph{permuted triangle inequalities}: for any $x,y,z\in X$ and $f\in
A(x,y)$, $g\in A(y,z)$ and $h\in A(x,z)$, we should have 
\begin{equation*}
\alpha (f)\leq \alpha (g)+\alpha (h)+d(f,g,h)
\end{equation*}%
and 
\begin{equation*}
\alpha (g)\leq \alpha (f)+\alpha (h)+d(f,g,h)\text{.}
\end{equation*}

\begin{lemma}
Under the above axioms, we have $\alpha (f)\geq 0$ for all $f\in A(x,y)$.
\end{lemma}

\begin{proof}
Let $g:=1_{y}$ and $h:=f$, then 
\begin{equation*}
d(f,g,h)=d(f,1_{y},f)=0
\end{equation*}
by the left identity property of $d$, so 
\begin{equation*}
0=\alpha (g)\leq \alpha (f)+\alpha (h)+d(f,g,h)=\alpha (f).
\end{equation*}
\end{proof}

\begin{lemma}
Suppose $f,f^{\prime }\in A(x,y)$. Then 
\begin{equation*}
\alpha (f^{\prime })\leq \alpha (f)+\phi (f,f^{\prime }),
\end{equation*}%
implying that 
\begin{equation*}
|\alpha (f)-\alpha (f^{\prime })|\leq \phi (f,f^{\prime }).
\end{equation*}
In particular, $\alpha $ is continuous.
\end{lemma}

\begin{proof}
Apply the main property of $\alpha $ with $g:=1_{y}$ and $h:=f^{\prime }$.
It says 
\begin{equation*}
\alpha (f^{\prime })\leq \alpha (f)+\alpha (1_{y})+d(f,1_{y},f^{\prime }).
\end{equation*}%
Since $\alpha (1_{y})=0$ by hypothesis and 
\begin{equation*}
d(f,1_{y},f^{\prime })=\phi (f,f^{\prime })\text{,}
\end{equation*}
we get the first statement. The rest follows.
\end{proof}

We might want to assume the \emph{anti-reflexivity property} that%
\begin{equation*}
\alpha (f)=0\Rightarrow x=y,\;\;f=1_{x}\text{.}
\end{equation*}
However, we might also want to ignore this property for example if we set 
$\alpha (f):=1$ for all $f$.

\section{Transitivity}
\label{sec-trans}

Suppose $\alpha$ is a function on the arrow sets. We define a notion of 
\emph{transitivity}, or rather in fact several related notions. Here by
convention an $\inf$ over the empty set is $+\infty$ and its product with $0$
is said to be $0$.

\begin{definition}
\label{abstrans} We say that $(X,A,d)$ satisfies \emph{left transitivity
with respect to $\alpha$} if for all $x,y,z,w\in X$ and $f\in A(x,y)$, $g\in
A(y,z)$, $h\in A(z,w)$, $k\in A(y,w)$ and $l\in A(x,w)$ we have 
\begin{equation*}
\alpha (k)\inf _{a\in A(x,z)}(d(f,g,a) + d(a,h,l)) \leq d(g,h,k) + d(f,k,l). 
\end{equation*}
We say that $(X,A,d)$ satisfies \emph{right transitivity with respect to $%
\alpha$} if for all $x,y,z,w\in X$ and $f\in A(x,y)$, $g\in A(y,z)$, $h\in
A(z,w)$, $k\in A(x,z)$ and $l\in A(x,w)$ we have 
\begin{equation*}
\alpha (k)\inf _{a\in A(y,w)}(d(g,h,a) + d(f,a,l)) \leq d(f,g,k) + d(k,h,l). 
\end{equation*}
We say that $(X,A,d)$ is \emph{$\alpha$-transitive} if it satisfies both
conditions.

We say that $(X,A,d)$ is \emph{absolutely (left or right) transitive} if it
satisfies one or both of the above conditions with $\alpha (k)=1$ for all $k$%
.
\end{definition}

These notions may be motivated by the transitivity condition introduced in 
\cite{AS1} as shown by the following lemma. Another motivation is that these
transitivity conditions provide the information to be used in Section \ref%
{yoneda} below in order to show that the Yoneda constructions give functors.

\begin{lemma}
Suppose $(X,d_{X})$ is a set with a function $d_{X}(x,y,z)\in {\mathbb{R}}$.
Let $A^{\mathrm{coarse}}(x,y):=\{\ast _{x,y}\}$ define the coarse graph
structure on $X$. Define 
\begin{equation*}
d(\ast _{x,y},\ast _{y,z},\ast _{x,z}):=d_{X}(x,y,z)\text{.}
\end{equation*}
Then $(X,A^{\mathrm{coarse}},d)$ is an AC structure if and only if $(X,d_{X})
$ is a $2$-metric space. Suppose the $2$-metric $d_{X}$ is bounded. Put $%
\alpha (\ast _{x,y}):=\varphi (x,y)$ where $\varphi (x,y):=\sup_{c\in
X}d(x,y,c)$ is the distance function \cite{AS1}. If $(X,d_{X})$ satisfies
the transitivity axiom of \cite{AS1}, then $(X,A^{\mathrm{coarse}},d)$ is an 
$\alpha /2$-transitive AC structure. If $(X,A^{\mathrm{coarse}},d)$ is an $%
\alpha $-transitive AC structure then $(X,d_{X})$ satisfies the transitivity
axiom of \cite{AS1}.
\end{lemma}

\begin{proof}
Suppose $(X,d_{X})$ is a $2$-metric space, then we obtain an AC structure.
The identities are $1_{x}:=\ast _{x,x}$. We have $d_{X}(x,x,y)=0$ and $%
d_{X}(x,y,y)=0$, which show the left and right identity axioms for an AC
structure. The left associativity property for an AC structure requires that
for any $x,y,z,w\in X$ we have 
\begin{eqnarray*}
d(\ast _{x,z},\ast _{z,w},\ast _{x,w}) &\leq &d(\ast _{x,y},\ast _{y,x},\ast
_{x,z})+d(\ast _{y,z},\ast _{z,w},\ast _{y,w}) \\
&&+d(\ast _{x,y},\ast _{y,w},\ast _{x,w}).
\end{eqnarray*}%
This translates as 
\begin{equation*}
d_{X}(x,z,w)\leq d_{X}(x,y,z)+d_{X}(y,z,w)+d_{X}(x,y,w)
\end{equation*}%
which is the tetrahedral axiom for a $2$-metric space with $y$ as the point
in the middle. Similarly, right associativity for the AC structure
translates to the same tetrahedral axiom but with $z$ as the point in the
middle. Thus if $(X,d_{X})$ is a $2$-metric space then $(X,A,d)$ is an
AC-structure.

In the other direction, if $(X,A,d)$ is an AC structure then we have seen
that $d(\ast _{x,y},\ast_{y,z}, \ast _{x,z})\geq 0$, so $d_X(x,y,z)\geq 0$.
The axioms for a $2$-metric space now translate from the axioms for an AC
structure as above.

We now relate the transitivity conditions. This is not a perfect
correspondence, because we have modified our definition of transitivity
slightly in order that it work better with the discussion to come later in
the paper. Suppose $(X,d_{X})$ is a bounded $2$-metric space and put%
\begin{equation*}
\alpha (\ast _{x,y}):=\varphi (x,y)=\sup_{c\in X}d_{X}(x,y,c)\text{.}
\end{equation*}

The transitivity axiom of \cite{AS1} says that given $4$ points $x,y,z,w$ we
have 
\begin{equation*}
d_{X}(x,y,w)\varphi (y,z)\leq d_{X}(x,y,z)+d_{X}(y,z,w).
\end{equation*}%
It implies by permutation that 
\begin{equation*}
d_{X}(x,z,w)\varphi (y,z)\leq d_{X}(x,y,z)+d_{X}(y,z,w).
\end{equation*}

On the other hand, for an amplitude $\alpha $ our left and right
transitivity axioms in Definition \ref{abstrans} both translate in terms of $%
d_{X}$ to 
\begin{equation*}
\alpha (\ast _{y,z})(d_{X}(x,y,w)+d_{X}(x,z,w))\leq
d_{X}(x,y,z)+d_{X}(y,z,w).
\end{equation*}%
In the case of left transitivity we should apply Definition \ref{abstrans}
to the points in order $y,x,z,w$ which is the same as right transitivity of
Definition \ref{abstrans} for the points in order $x,y,w,z$.

If we assume the transitivity of \cite{AS1} then by adding the two previous
equations and dividing by $2$ we get Definition \ref{abstrans} for the
function $\alpha /2$. On the other hand, by positivity of the distances, if
we know the condition of Definition \ref{abstrans} for $\alpha$ then we get
the transitivity property of \cite{AS1}.
\end{proof}

Notice that the function $\alpha$ used for the definition of absolutely
transitive, is not an amplitude since it doesn't satisfy the property $%
\alpha (1_x)=0$. It will be interesting to see if Theorems \ref{mainintro}
and \ref{main}, which use the absolute transitivity condition in an
essential way, can be extended to relative transitivity with respect to an
amplitude function $\alpha$.

\section{The $0$-categoric situation}

If $(X,A,d)$ is $0$-categoric, then it corresponds to an actual category,
and the composition is a non expansive function $A(x,y)\times
A(y,z)\rightarrow A(x,z)$ with respect to the sum distance on the product.

\begin{theorem}
\label{zerocat}  Suppose $(X,A,d)$ is a $0$-categoric approximate category,
and suppose that it is separated (Definition \ref{separated}). Then for any $%
f\in A(x,y)$ and $g\in A(y,z)$ there is a unique element denoted $g\circ
f\in A(x,z)$ such that $d(f,g,g\circ f)=0$. This defines a composition
operation making $(X,A,\circ ,\phi )$ into a metrized category. If we let $%
d_{\phi}(f,g,h):= \phi (g\circ f, h)$ then we have $d_{\phi}(f,g,h)\leq
d(f,g,h)$ whenever these are defined. The composition maps 
\begin{equation*}
A(x,y)\times A(y,z)\rightarrow A(x,z) 
\end{equation*}
are continuous, and indeed they are distance nonincreasing if the left hand
side is provided with the sum metric.
\end{theorem}

\begin{proof}
By Corollary \ref{phizero}, if $h$ and $h^{\prime }$ are any elements such
that%
\begin{equation*}
d(f,g,h)=d(f,g,h^{\prime })=0\text{,}
\end{equation*}
then $\phi (h,h^{\prime })=0$. Since $(X,A,d)$ is separated, this implies
that $h=h^{\prime }$. Therefore, the composition $h=g\circ f$ is unique. The
associativity (resp. unit) properties imply that the composition is
associative (resp. has units).

We would now like to bound the norm of the composition operation. Suppose $%
f,f^{\prime }\in A(x,y)$, $g,g^{\prime }\in A(y,z)$ and let $h:=g\circ f$
and $h^{\prime }:=g^{\prime }\circ f^{\prime }$. Apply Corollary \ref%
{dphicorollary} to $f^{\prime },f,g^{\prime },g$, and two times $h^{\prime }$%
. As $d(f^{\prime },g^{\prime },h^{\prime })=0$ and $\phi (h^{\prime
},h^{\prime })=0$ we get 
\begin{equation*}
d(f,g,h^{\prime })\leq \phi (f,f^{\prime })+\phi (g,g^{\prime }).
\end{equation*}%
On the other hand, 
\begin{equation*}
\phi (h,h^{\prime })=d(h,1_{z},h^{\prime }).
\end{equation*}%
Applying the the associativity tetrahedral property to $f,g,1_{z};h,g;h^{%
\prime }$ we get 
\begin{equation*}
d(h,1_{z},h^{\prime })\leq d(f,g,h)+d(g,1_{z},g)+d(f,g,h^{\prime }).
\end{equation*}%
This gives 
\begin{equation*}
\phi (h,h^{\prime })\leq \phi (f,f^{\prime })+\phi (g,g^{\prime }).
\end{equation*}%
It says that the composition map is non increasing from the sum distance on $%
A(x,y)\times A(y,z)$ to $A(x,z)$.
\end{proof}

\begin{proposition}
Suppose that $(X,A,d)$ is $\epsilon$-categoric for any $\epsilon > 0$ and
each metric space $(A(x,y),\phi )$ is complete. Then it is $0$-categoric.
\end{proposition}

\begin{proof}
Given $f\in A(x,y)$ and $g\in A(y,z)$, for every positive integer $m$,
choose an $h_{m}$ such that 
\begin{equation*}
d(f,g,h_{m})\leq 1/m\text{.}
\end{equation*}
By left associativity for $f,g,1_{z};h_{m},g;h_{n}$ we have 
\begin{eqnarray*}
\phi (h_{m},h_{n}) &=&d(h_{m},1_{z},h_{n}) \\
&\leq &d(f,g,h_{m})+d(g,1_{z},g)+d(f,g,h_{n}) \\
&\leq &\frac{1}{m}+\frac{1}{n}\text{.}
\end{eqnarray*}%
It follows that $(h_{m})$ is a Cauchy sequence. By the completeness
hypothesis, it has a limit which we denote $g\circ f$. By left associativity
for $1_{x},f,g;f,h_{m};g\circ f$ we get 
\begin{eqnarray*}
d(f,g,g\circ f) &\leq &d(1_{x},f,f)+d(f,g,h_{m})+d(1_{x},h_{m},g\circ f) \\
&\leq &\frac{1}{m}+\phi (h_{m},g\circ f).
\end{eqnarray*}%
The right side $\rightarrow 0$ as $m\rightarrow \infty $ so we obtain $%
d(f,g,g\circ f)=0$. This is the $0$-categoric property.
\end{proof}

\begin{lemma}
\label{epcattrans} If $(X,A,d,\alpha )$ is $\epsilon$-categoric for any $%
\epsilon >0$, then it is absolutely transitive.
\end{lemma}

\begin{proof}
We show absolute left transitivity. Suppose given $f,g,h,k,l$ as in the
definition. For any $\epsilon >0$ there exists $a\in A(x,z)$ such that $%
d(f,g,a)<\epsilon $. By the tetrahedral axiom, 
\begin{eqnarray*}
d(a,h,l) &\leq &d(f,g,a)+d(g,h,k)+d(f,k,l) \\
&=&\epsilon +d(g,h,k)+d(f,k,l).
\end{eqnarray*}%
Therefore 
\begin{equation*}
d(f,g,a)+d(a,h,l)\leq 2\epsilon +d(g,h,k)+d(f,k,l).
\end{equation*}%
Such an $a$ exists for any $\epsilon >0$, thus 
\begin{equation*}
\inf_{a\in A(x,z)}(d(f,g,a)+d(a,h,l))\leq d(g,h,k)+d(f,k,l).
\end{equation*}%
This is the absolute left transitivity condition. The proof for absolute
right transitivity is similar.
\end{proof}

\section{Examples}

\label{examples}

\subsection{Example from a $2$-metric space}

\label{ex2met}

Suppose $(X,d)$ is a bounded $2$-metric space. That is, we suppose first
that $d$ satisfies the axioms denoted $\mathrm{(Sym)}$, $\mathrm{(Tetr)}$, $%
\mathrm{(Z)}$, $\mathrm{(N)}$, and $\mathrm{(B)}$ of \cite{AS1}. As there,
define $\varphi (x,y)$ to be the supremum of $d(x,y,c)$ for $c\in X$. Then
we put $A(x,y):=\{f_{x,y}\}$ (the set with a single element which is denoted
as $f_{x,y}$). Set $1_{x}:=f_{x,x}$. Define 
\begin{equation*}
d(f_{x,y},f_{y,z},f_{x,z}):=d(x,y,z)\text{ and }\alpha (f_{x,y}):=\varphi
(x,y)\text{.}
\end{equation*}

\begin{lemma}
With the above notations and hypotheses, the structure $(X,A,d,\alpha )$ is
an approximate categorical structure with amplitude $\alpha $ and it
satisfies the small transitivity condition. If $(X,d)$ satisfies
transitivity (\cite{AS1} axiom $\mathrm{(Trans)}$, then the approximate
categorical structure is transitive with respect to the amplitude $\alpha
=\varphi $ in the sense of Definition \ref{abstrans}.
\end{lemma}

\begin{proof}
The associativity conditions come from $\mathrm{(Tetr)}$ using symmetry $%
\mathrm{(Sym)}$. The identity conditions with $1_{x}:=f_{x,x}$ come from $%
\mathrm{(Z)}$. The axioms for $\alpha (f_{x,y}):=\varphi (x,y)$ come from
the property%
\begin{equation*}
\varphi (x,z)\leq \varphi (x,y)+\varphi (y,z)+d(x,y,z)\text{,}
\end{equation*}
which in turn came from the tetrahedral condition in view of the definition $%
\varphi (x,y):=\mathrm{sup}_{c}d(x,y,c)$. Small transitivity comes about
because%
\begin{equation*}
d(x,y,z)\leq \mathrm{min}\{\alpha (f_{x,y}),\alpha (f_{y,z})\}
\end{equation*}
in view of the definition of $\varphi $. Similarly, transitivity of the $2$%
-metric implies transitivity for the approximate categorical structure.
\end{proof}

\noindent \emph{Remark:} If $L\subset X$ is a line, then it corresponds to a 
$0$-categoric sub-structure of $(X,A,d)$.

\noindent \emph{Remark:} The AC structure $(X,A,d)$ defined from a $2$%
-metric space as above, is generally not absolutely transitive, because we
need to use the amplitude and in particular $\alpha (1_x)=0$.

\subsection{A finite example}

\label{exfinite}

Consider a very first case. Let $X=\{x\}$ have a single object and $%
A(x,x)=\{1,e\}$ with $1=1_{x}$. Put 
\begin{equation*}
\phi :=\phi (1,e)=d(1,e,1)=d(e,1,1)=d(1,1,e).
\end{equation*}%
The remaining quantities to consider are $d(e,e,e)$ and $d(e,e,1)$. Recall
that there are two categorical structures, with $e^{2}=e$ or $e^{2}=1$ and
these two numbers represent the distances to these two cases.

From the various associativity laws we get the following inequalities: 
\begin{equation*}
d(e,e,e)\leq \phi ,
\end{equation*}%
\begin{equation*}
d(e,e,1)\leq 2\phi ,
\end{equation*}%
\begin{equation*}
|d(e,e,e)-\phi |\leq d(e,e,1),
\end{equation*}%
and 
\begin{equation*}
|d(e,e,1)-\phi |\leq d(e,e,e).
\end{equation*}%
Since everything is invariant under scaling (and trivial if $\phi =0$) we
may assume $\phi =1$ and set 
\begin{equation*}
u:=d(e,e,e),\;\;\;v:=d(e,e,1).
\end{equation*}%
Note that $u,v\geq 0$. The inequalities become 
\begin{equation*}
u\leq 1,\text{ }v\leq 2,\text{ }|u-1|\leq v\text{ \ and \ }|v-1|\leq u
\end{equation*}%
which reduce to 
\begin{equation*}
u\leq 1,\text{ }u+v\geq 1\text{ \ and \ }v\leq u+1\text{.}
\end{equation*}%
Hence, the graph of the allowed region in the $(u,v)$-plane looks like: 
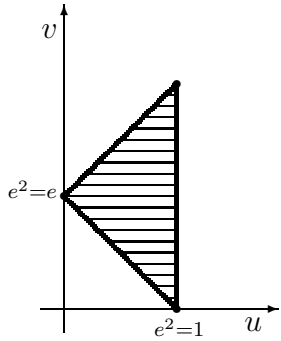
\begin{equation*}
\setlength{\unitlength}{.3mm}%
\begin{picture}(150,200)


\qbezier(40,30)(40,100)(40,170)
\put(40,170){\vector(0,1){5}}
\put(30,160){$v$}

\qbezier(30,40)(80,40)(130,40)
\put(130,40){\vector(1,0){5}}
\put(120,30){$u$}

\linethickness{.5mm}
\qbezier(90,40)(90,90)(90,140)

\qbezier(90,40)(65,65)(40,90)

\qbezier(40,90)(65,115)(90,140)
\thinlines

\put(90,40){\circle*{4}}
\put(80,28){${\scriptstyle e^2=1}$}

\put(40,90){\circle*{4}}
\put(15,89){${\scriptstyle e^2=e}$}

\put(90,140){\circle*{4}}

\put(90,45){\line(-1,0){5}}
\put(90,50){\line(-1,0){10}}
\put(90,55){\line(-1,0){15}}
\put(90,60){\line(-1,0){20}}
\put(90,65){\line(-1,0){25}}
\put(90,70){\line(-1,0){30}}
\put(90,75){\line(-1,0){35}}
\put(90,80){\line(-1,0){40}}
\put(90,85){\line(-1,0){45}}
\put(90,90){\line(-1,0){50}}

\put(90,135){\line(-1,0){5}}
\put(90,130){\line(-1,0){10}}
\put(90,125){\line(-1,0){15}}
\put(90,120){\line(-1,0){20}}
\put(90,115){\line(-1,0){25}}
\put(90,110){\line(-1,0){30}}
\put(90,105){\line(-1,0){35}}
\put(90,100){\line(-1,0){40}}
\put(90,95){\line(-1,0){45}}

\end{picture}
\end{equation*}%
The categorical structures are $(u,v)=(1,0)$ for $e^{2}=1$ and $(u,v)=(0,1)$
for $e^{2}=e$. The third vertex $(1,2)$ is an extremal case where no
categorical relations hold.

\subsection{Paths}

\label{expaths}

Consider $X:= {\mathbb{R}} ^2$, and let $A(x,y)$ be the set of continuous
paths $f:[0,1]\rightarrow X$ with $f(0)=x$ and $f(1)=y$. Let $d(f,g,h)$
denote the infimum of the areas of disks mapping to $X$ such that the
boundary maps to the circle defined by joining the paths $f$, $g$ and $h$.
Let $1_x$ denote the constant path at $x$.

\begin{lemma}
The resulting triple $(X,A,d)$ is an approximate categorical structure.
\end{lemma}

\begin{proof}
Suppose $f:[0,1]\rightarrow X$ is a path from $x$ to $y$. If we define $%
\varphi (s,t):=f(s)$ and restrict to the triangle whose vertices are $(0,0)$%
, $(1,0)$ and $(1,1)$. The triangle is homeomorphic to a disk and we obtain
a disk mapping to $X$ whose boundary consists of the paths $f$, $1_{y}$ and $%
f$ again with total area zero. This shows $d(f,1_{y},f)=0$. The other
identity axiom holds similarly. For the tetrahedral axioms, given three
disks corresponding to triangles in the interior of the tetrahedron we can
paste them together to get a disk whose boundary consists of the three outer
edges and whose area is the sum of the three areas. This shows the required
tetrahedral property for either left or right associativity.
\end{proof}

\section{Functors }

\label{functors}

Given $(X,A)$ and $(Y,B)$, a \emph{prefunctorial map} $F: (X,A)\rightarrow
(Y,B)$ consists of a map $F: X\rightarrow Y$ and, for all $x,y\in X$, a map $%
F: A(x,y)\rightarrow A(Fx, Fy)$. We generally assume that it is \emph{unital}%
, that is $F(1_x)= 1_{Fx}$.

Given approximate categorical structures denoted $d$ on $(X,A)$ and $(Y,B)$
and $k\geq 0$, we say that a prefunctorial map $F$ is \emph{$k$-functorial}
if it is unital and whenever $x,y,z\in X$ and $f\in A(x,y)$, $g\in A(y,z)$
and $h\in A(x,z)$ we have 
\begin{equation*}
d(F(f),F(g),F(h))\leq kd(f,g,h).
\end{equation*}

If $(X,A)$ and $(Y,B)$ are $0$-categorical and separated, then this implies
that $F$ respects composition, hence it defines a functor between categories.

\begin{lemma}
Suppose $F: (X,A)\rightarrow (Y,B)$ is a $k$-functorial map. Then for any $%
x,y\in X$ and $f,f^{\prime }\in A(x,y)$ we have 
\begin{equation*}
\phi (F(f),F(f^{\prime }))\leq k\phi (f,f^{\prime }). 
\end{equation*}
\end{lemma}

\begin{proof}
It follows from the definition of $\phi$ and the condition that $F$ is
unital.
\end{proof}

\section{Metric correspondences}

\label{metriccorrespondences}

Suppose $(X,d_{X})$ and $(Y,d_{Y})$ are bounded metric spaces. We define a
set of \emph{metric correspondences} denoted ${\mathcal{M}}(X,Y)$ as
follows. An element of ${\mathcal{M}}(X,Y)$ is a bounded function $f:X\times
Y\rightarrow {\mathbb{R}}$ satisfying the following axioms: \newline
(MC0)---if $X$ is nonempty then $Y$ is nonempty; \newline
(MC1)---for any $x,x^{\prime }\in X$ and $y\in Y$ we have%
\begin{equation*}
f(x,y)\leq d_{X}(x,x^{\prime })+f(x^{\prime },y)\text{;}
\end{equation*}
(MC2)---for any $x\in X$ and $y,y^{\prime }\in Y$ we get%
\begin{equation*}
f(x,y)\leq f(x,y^{\prime })+d_{Y}(y,y^{\prime })\text{.}
\end{equation*}

A \emph{functional metric correspondence} is a metric correspondence which
also satisfies the axiom \newline
(MF)---for any $x\in X$ and $y,y^{\prime }\in Y$ we obtain%
\begin{equation*}
d_{Y}(y,y^{\prime })\leq f(x,y)+f(x,y^{\prime })\text{.}
\end{equation*}

Let ${\mathcal{F}}(X,Y)\subset {\mathcal{M}}(X,Y)$ be the subset of
functional metric correspondences.

\begin{definition}
If $X,Y,Z$ are metric spaces, and $f\in {\mathcal{M}}(X,Y)$ and $g\in {%
\mathcal{M}}(Y,Z)$ define their composition by 
\begin{equation*}
(g\circ f)(x,z):= \inf _{y\in Y} ( f(x,y)+ g(y,z)) . 
\end{equation*}
If $X\neq \emptyset$ then by (MC0) also $Y\neq \emptyset$ so we can form the 
$\inf$. If $X=\emptyset$ then nothing needs to be given to define $(g\circ f)
$.

Define the identity $i_{X}\in {\mathcal{M}}(X,X)$ by 
\begin{equation*}
i_{X}(x,x):=d_{X}(x,x)\text{.}
\end{equation*}

Define a distance on ${\mathcal{M}}(X,Y)$ by 
\begin{equation*}
d_{\mathcal{M}}(X,Y)(f,f^{\prime }):= \sup _{x\in X,y\in Y}|
f(x,y)-f^{\prime }(x,y)| . 
\end{equation*}
The supremum exists since we have assumed that our correspondence function $f
$ is bounded.
\end{definition}

\begin{proposition}
\label{mcmc} The composition operation 
\begin{equation*}
\circ : {\mathcal{M}}(Y,Z)\times {\mathcal{M}}(X,Y) \rightarrow {\mathcal{M}}%
(X,Z) 
\end{equation*}
defined in the previous definition, with the identities $i_X$, provides a
structure of metrized category denoted $\mathbf{MetCor}$ whose objects are
the metric spaces and whose morphism spaces are the metric spaces \linebreak[%
9] $({\mathcal{M}}(X,Y), d_{{\mathcal{M}}(X,Y)})$.
\end{proposition}

\begin{proof}
First, suppose $f\in {\mathcal{M}}(X,Y)$, and consider the composition $%
g:=f\circ i_{X}$. We have 
\begin{equation*}
g(x,y)=\inf_{u\in X}(i_{X}(x,u)+f(u,z)).
\end{equation*}%
Taking $u:=x$ we get $g(x,y)\leq f(x,y)$, but on the other hand, by
hypothesis 
\begin{equation*}
f(x,y)\leq i_{X}(x,u)+f(u,y)
\end{equation*}%
for any $u$, so 
\begin{equation*}
f(x,y)\leq g(x,y)\text{.}
\end{equation*}
This shows the right identity axiom $f\circ i_{X}=f$ and the proof for left
identity is the same.

Suppose $f\in {\mathcal{M}}(X,Y)$, $g\in {\mathcal{M}}(Y,Z)$ and $h\in {%
\mathcal{M}}(Z,W)$. Put $a:=g\circ f$. Then 
\begin{equation*}
a(x,z)=\inf_{y\in Y}(f(x,y)+g(y,z))
\end{equation*}%
and 
\begin{eqnarray*}
(h\circ a)(x,w) &=&\inf_{z\in Z}(a(x,z)+h(z,w)) \\
&=&\inf_{y\in Y,z\in Z}(f(x,y)+g(y,z)+h(z,w)).
\end{eqnarray*}%
If we put $b:=h\circ g$ then the expression for $(b\circ f)(x,w)$ is the
same, showing associativity. The composition operation therefore defines a
category.

To show that the metric gives a metrized structure, we need to show that 
\begin{equation*}
d_{\mathcal{M}}(X,Z)(g\circ f,g\circ f^{\prime })\leq d_{\mathcal{M}%
}(X,Y)(f,f^{\prime })+d_{\mathcal{M}}(Y,Z)(g,g^{\prime }).
\end{equation*}%
Suppose 
\begin{equation*}
d_{\mathcal{M}}(X,Y)(f,f^{\prime })<\epsilon \text{ \ and \ }d_{\mathcal{M}%
}(Y,Z)(g,g^{\prime })<\varepsilon \text{.}
\end{equation*}%
It means that 
\begin{equation*}
f(x,y)<f^{\prime }(x,y)+\epsilon \text{, }\;f^{\prime }(x,y)<f(x,y)+\epsilon 
\end{equation*}%
and 
\begin{equation*}
g(y,z)<g^{\prime }(y,z)+\varepsilon \text{,}\;\;g^{\prime
}(y,z)<g(y,z)+\varepsilon .
\end{equation*}%
Then 
\begin{eqnarray*}
(g^{\prime }\circ f^{\prime })(x,z) &=&\inf_{y\in Y}(f^{\prime
}(x,y)+g^{\prime }(y,z)) \\
&\leq &\inf_{y\in Y}(f(x,y)+\epsilon +g(y,z)+\varepsilon ) \\
&=&(g\circ f)(x,z)+\epsilon +\varepsilon .
\end{eqnarray*}%
Similarly%
\begin{equation*}
(g\circ f)(x,z)\leq (g^{\prime }\circ f^{\prime })(x,z)+\epsilon
+\varepsilon \text{.}
\end{equation*}%
It follows from this statement that 
\begin{equation*}
d_{\mathcal{M}}(X,Z)(g\circ f,g\circ f^{\prime })\leq d_{\mathcal{M}%
}(X,Y)(f,f^{\prime })+d_{\mathcal{M}}(Y,Z)(g,g^{\prime })
\end{equation*}%
as required.
\end{proof}

We can also provide the collection of sets ${\mathcal{M}}(X,Y)$ with an AC
structure. If $X,Y,Z$ are three metric spaces, define for $f\in {\mathcal{M}}%
(X,Y)$, $g\in {\mathcal{M}}(Y,Z)$ and $h\in {\mathcal{M}}(X,Z)$ the distance 
\begin{equation*}
d(f,g,h):= \sup _{x\in X, z\in Z} \left| h(x,z) - \inf _{y\in Y}
(f(x,y)+g(y,z)) \right| . 
\end{equation*}
Write $X\overset{f}{\dashrightarrow} Y$ if $f\in {\mathcal{M}}(X,Y)$.

\begin{lemma}
The condition $d(f,g,h)<\epsilon $ is equivalent to the conjunction of the
following two conditions: \newline
(d1)---for any $x,y,z$ we have 
\begin{equation*}
h(x,z)\leq f(x,y)+g(y,z)+\epsilon 
\end{equation*}
and \newline
(d2)---for any $x,z$ there exists $y\in Y$ with%
\begin{equation*}
f(x,y)+g(y,z)\leq h(x,z)+\epsilon \text{.}
\end{equation*}
\end{lemma}

\begin{proof}
This is similar to the technique used in the previous proof.
\end{proof}

\begin{corollary}
The above distance satisfies the axioms for an AC structure. Furthermore, it
is absolutely transitive. It is the AC structure associated to the metrized
category $\mathbf{MetCor}$.
\end{corollary}

\begin{proof}
This follows from Propositions \ref{mcmc} and \ref{mcac}. Absolute
transitivity follows from Lemma \ref{epcattrans}.
\end{proof}

We can more generally define, for any $k>0$, the $k$-contractive metric
correspondences ${\mathcal{M}}(X,Y;k)$. For this, we keep the second
condition the same but modify the first condition so it says \newline
(MC1)---for any $x,x^{\prime }\in X$ and $y\in Y$ we have%
\begin{equation*}
f(x,y)\leq kd_{X}(x,x^{\prime })+f(x^{\prime },y)\text{;}
\end{equation*}
\newline
(MC2)---for any $x\in X$ and $y,y^{\prime }\in Y$ we get%
\begin{equation*}
f(x,y)\leq f(x,y^{\prime })+d_{Y}(y,y^{\prime })\text{.}
\end{equation*}

Again, the functionality condition (MF) is the same as before. Notice that
the identity $i_{X}$ will be in here only if $k\geq 1$ and furthermore if $f$
is $k$-contractive then we would need $k\leq 1$ in order to get $%
d(i_{X},f,f)=0$.

It will undoubtedly be interesting to try to iterate the composition (to be
defined in the proof of transitivity above) of $k$-contractive
correspondences and to study convergence of the iterates.

\section{The Yoneda functors}

\label{yoneda}

Suppose $(X,A,d)$ is an AC structure, such that each $A(x,y)$ is a bounded
metric space. Choose $u\in X$. Then we would like to define a
\textquotedblleft Yoneda functor\textquotedblright\ $x\mapsto A(u,x)$ from $%
(X,A,d)$ to the AC structure of metric correspondences. Put 
\begin{equation*}
Y_{u}(x):=(A(u,x),\mathrm{dist}_{A(u,x)})
\end{equation*}%
where the distance $\mathrm{dist}_{A(u,x)}$ is the distance $\phi $ coming
from $d$ as in Section \ref{metarrowsets}. We assume the separation axiom of
Definition \ref{separated}, so $Y_{u}(x)$ is a metric space and it is
bounded by assumption.

Make the following hypothesis.

\begin{hypothesis}
\label{graphcomp} Whenever $A(x,y)$ and $A(y,z)$ are nonempty, then $A(x,z)$
is nonempty too.
\end{hypothesis}

For any $f\in A(x,y)$ define $Y_u(f)\in {\mathcal{M}}(Y_u(x), Y_u(y))$ by 
\begin{equation*}
Y_u(f) (a,b):= d(a,f,b). 
\end{equation*}

\begin{lemma}
Assuming Hypothesis \ref{graphcomp}, if $f\in A(x,y)$ then%
\begin{equation*}
Y_{u}(f)\in {\mathcal{M}}(Y_{u}(x),Y_{u}(y))\text{.}
\end{equation*}
\end{lemma}

\begin{proof}
We need to show (MC0), (MC1) and (MC2). Suppose $Y_u(x)=A(u,x)$ is nonempty.
Then by Hypothesis \ref{graphcomp}, $Y_u(y)=A(u,y)$ is also nonempty, giving
(MC0).

Suppose $a,a^{\prime }\in Y_{u}(x)=A(u,x)$ and $b\in Y_{u}(y)=A(u,y)$. We
have 
\begin{eqnarray*}
Y_{u}(f)(a,b) &=&d(a,f,b)\leq d_{A(u,x)}(a,a^{\prime })+d(a^{\prime },f,b) \\
&=&d_{A(u,x)}(a,a^{\prime })+Y_{u}(f)(a^{\prime },b)
\end{eqnarray*}%
by Lemma \ref{phileft}, giving (MC1).

Suppose $a\in Y_{u}(x)=A(u,x)$ and $b,b^{\prime }\in Y_{u}(y)=A(u,y)$, then 
\begin{eqnarray*}
Y_{u}(f)(a,b) &=&d(a,f,b)\leq d_{A(u,y)}(b,b^{\prime })+d(a,f,b^{\prime }) \\
&=&d_{A(u,y)}(b,b^{\prime })+Y_{u}(f)(a,b^{\prime })
\end{eqnarray*}%
by Lemma \ref{phicomp}, giving (MC2).
\end{proof}

\begin{proposition}
\label{abstransYoneda} Suppose $(X,A,d)$ satisfies absolute left
transitivity (Definition \ref{abstrans}). Then the Yoneda map $Y_u$ defined
above is a functor (with contractivity constant $1$).
\end{proposition}

\begin{proof}
Suppose $x,y,z\in X$ and $f\in A(x,y)$, $g\in A(y,z)$ and $h\in A(x,z)$. We
would like to show 
\begin{equation*}
d_{{\mathcal{M}}(x,z)}(Y_{u}(g)\circ Y_{u}(f),Y_{u}(h))\leq d(f,g,h).
\end{equation*}%
We have for $a\in Y_{u}(x)=A(u,x)$ and $c\in Y_{u}(z)=A(u,z)$, 
\begin{equation*}
Y_{u}(g)\circ Y_{u}(f)(a,c)=\inf_{b\in A(u,y)}(Y_{u}(g)(b,c)+Y_{u}(f)(a,b)).
\end{equation*}%
Note that by Hypothesis \ref{graphcomp}, given $a$ and $f$ it follows that $%
A(u,y)$ is nonempty so the $\inf $ exists. Now 
\begin{equation*}
d_{{\mathcal{M}}(x,z)}(Y_{u}(g)\circ Y_{u}(f),Y_{u}(h))=
\end{equation*}%
\begin{equation*}
\sup_{a\in A(u,x),c\in A(u,z)}\left\vert Y_{u}(h)(a,c)-Y_{u}(g)\circ
Y_{u}(f)(a,c)\right\vert 
\end{equation*}%
\begin{equation*}
=\sup_{a\in A(u,x),c\in A(u,z)}\left\vert Y_{u}(h)(a,c)-\inf_{b\in
A(u,y)}(Y_{u}(g)(b,c)+Y_{u}(f)(a,b))\right\vert 
\end{equation*}%
\begin{equation*}
=\sup_{a\in A(u,x),c\in A(u,z)}\left\vert d(a,h,c)-\inf_{b\in
A(u,y)}(d(b,g,c)+d(a,f,b))\right\vert .
\end{equation*}%
We would like to show that this is $\leq d(f,g,h)$. This is equivalent to
asking that for all $a\in A(u,x)$ and $c\in A(u,z)$ we should have 
\begin{equation}
d(a,h,c)-\inf_{b\in A(u,y)}(d(b,g,c)+d(a,f,b))\leq d(f,g,h)  \label{y1}
\end{equation}%
and 
\begin{equation}
\inf_{b\in A(u,y)}(d(b,g,c)+d(a,f,b))-d(a,h,c)\leq d(f,g,h).  \label{y2}
\end{equation}%
In turn, the first one \eqref{y1} is equivalent to 
\begin{equation*}
d(a,h,c)\leq d(f,g,h)+\inf_{b\in A(u,y)}(d(b,g,c)+d(a,f,b))
\end{equation*}%
and this is true by the tetrahedral inequality 
\begin{equation*}
d(a,h,c)\leq d(f,g,h)+d(b,g,c)+d(a,f,b)
\end{equation*}%
for any $b$. The second one \eqref{y2} is equivalent to 
\begin{equation*}
\inf_{b\in A(u,y)}(d(b,g,c)+d(a,f,b))\leq d(f,g,h)+d(a,h,c),
\end{equation*}%
but that is exactly the statement of the absolute left transitivity
condition of Definition \ref{abstrans}. Thus under our hypothesis, \eqref{y2}
is true. We obtain the required inequality.

For the identities, we need to know that $Y_u(1_x)= i_{Y_u(x)}$. Recall that
the identity $i_{Y_u(x)}$ in ${\mathcal{M}}( Y_u(x),Y_u(x))$ is just the
distance function $d_{Y_u(x)}$, and $Y_u(x)=A(u,x)$. Its distance function
is 
\begin{equation*}
d_{Y_u(x)} (f,f^{\prime })=d(f,1_x,f^{\prime }) 
\end{equation*}
by the discussion of Section \ref{metarrowsets}, and in turn this is exactly 
$Y_u(1_x)$. This shows that $Y_u$ preserves identities, and completes the
proof that $Y_u$ is a functor with contractivity constant $1$.
\end{proof}

We can similarly define Yoneda functors in the other direction 
\begin{equation*}
Y^{u}(x):=A(x,u)
\end{equation*}
with the same properties. The opposed statement of the previous proposition
says

\begin{proposition}
Suppose $(X,A,d)$ satisfies absolute right transitivity (Definition \ref%
{abstrans}). Then the Yoneda map $Y^u$ is a functor (with contractivity
constant $1$).
\end{proposition}

The proof is similar.

\section{Functors to metrized categories}

\label{fmc}

Suppose $(X,A,d)$ is an AC structure. We would like to look at functors $%
F:(X,A,d)\rightarrow ({{\mathcal{C}}} , d_{{{\mathcal{C}}}} )$ to metrized
categories.

First, we consider the \emph{free category} on $(X,A)$. Let $\mathbf{Free}%
(X,A)$ denote the free category on the graph $(X,A)$. Thus, the set of
objects of $\mathbf{Free}(X,A)$ is equal to $X$ and 
\begin{equation*}
\mathbf{Free}(X,A)(x,y):=
\end{equation*}%
\begin{equation*}
\{(x_{0},\ldots ,x_{k};a_{1},\ldots ,a_{k}):x_{i}\in
X,\;\;x_{0}=x,x_{k}=y,a_{i}\in A(x_{i-1},x_{i})\}.
\end{equation*}%
Arrows will be denoted just $(a_{1},\ldots ,a_{k})$ if there is no
confusion. Composition is by concatenation and the identity of $x$ in $%
\mathbf{Free}(X,A)$ is the sequence $()_{x}$ of length $k=0$ based at $%
x_{0}=x$.

Suppose given a functor $F:(X,A,d)\rightarrow ({{\mathcal{C}}},d_{{{\mathcal{%
C}}}})$. In particular, $F:X\rightarrow \mathrm{Ob}({{\mathcal{C}}})$ and
for any $x,y\in X$ we have $F:A(x,y)\rightarrow {{\mathcal{C}}}(Fx,Fy)$.
This induces a functor of usual categories 
\begin{equation*}
\mathbf{Free}(F):\mathbf{Free}(X,A)\rightarrow {{\mathcal{C}}}
\end{equation*}%
defined by sending a sequence $(a_{1},\ldots ,a_{k})$ to the composition $%
F(a_{k})\circ \cdots \circ F(a_{1})$ and sending $()_{x}$ to $1_{Fx}$.
Pulling back the distance on ${{\mathcal{C}}}$ we obtain a pseudometric $%
d_{F}$ on $\mathbf{Free}(X,A)$, which is to say for each $x,y$ we have a
distance $d_{F}$ on $\mathbf{Free}(X,A)(x,y)$. These distances form a
structure of pseudometrized category on $\mathbf{Free}(X,A)$, although the
separation axiom doesn't hold.

Now $F$ is a functor of AC structures, if and only if 
\begin{equation*}
d_{{{\mathcal{C}}}}(F(g)\circ F(f),F(h))\leq d(f,g,h)
\end{equation*}%
for any $f\in A(x,y)$, $g\in A(y,z)$ and $h\in A(x,z)$ and also if
\linebreak\ $d_{{{\mathcal{C}}}}(F(1_{x}),1_{Fx})=0$. These conditions are
equivalent to the following conditions on the pullback distance $d_{F}$: 
\begin{equation*}
d_{F}((f,g),(h))\leq d(f,g,h)\mbox{ and }d_{F}((1_{x}),()_{x})=0.
\end{equation*}%
Let $\mathrm{dfunct}(X,A,d)$ denote the set of pseudo-metric structures $%
d_{F}$ on $\mathbf{Free}(X,A)$ satisfying these conditions, we call these
elements \emph{functorial distances}.

\begin{proposition}
Assume Hypothesis \ref{graphcomp}. For any two arrows $a$ and $b$ in $%
\mathbf{Free}(X,A)(x,y)$, the set of values of $d_F(a,b)$ over all $d_F\in 
\mathrm{dfunct}(X,A,d)$ is bounded. Therefore we may set 
\begin{equation*}
d^{\mathrm{max}}(a,b):= \sup _{d_F\in \mathrm{dfunct}(X,A,d)} d_F(a,b), 
\end{equation*}
and $d^{\mathrm{max}}\in \mathrm{dfunct}(X,A,d)$ is the unique maximal
functorial distance.
\end{proposition}

\begin{proof}
Given that $\mathbf{Free}(X,A)(x,y)$ is non empty, Hypothesis \ref{graphcomp}
implies that $A(x,y)$ is nonempty, so we may fix some $h\in A(x,y)$. We have 
\begin{equation*}
d_{F}(a,b)\leq d_{F}(a,h)+d_{F}(h,b)\text{,}
\end{equation*}%
so it suffices to show that $d_{F}(a,h)$ is bounded (the case of $d_{F}(h,b)$
being the same by symmetry). Suppose $a$ is the composition of a sequence of
arrows $a_{i}\in A(x_{i-1},x_{i})$ with $x_{0}=x$ and $x_{n}=y$. Again by
our hypothesis, we may choose $h_{i}\in A(x_{0},x_{i})$ with $h_{n}=h$. We
show by induction on $i$ that $d_{F}((a_{i},\cdots ,a_{0}),h_{i})$ is
bounded as $F$ ranges over all functorial distances. This is true for $i=1$
since 
\begin{eqnarray*}
d_{F}(a_{1},h_{1}) &=&d_{F}((1_{x},a_{1}),h_{1}) \\
&\leq &d(1_{x},a_{1},h_{1})=\phi (a_{1},h_{1}).
\end{eqnarray*}%
Suppose it is known for $i-1$. Then 
\begin{eqnarray*}
d_{F}((a_{i},\cdots ,a_{0}),h_{i}) &=&d_{F}((a_{i},(a_{i-1},\cdots
,a_{0}),h_{i}) \\
&\leq &d_{F}((a_{i},h_{i-1}),h_{i})+d_{F}((a_{i-1},\cdots ,a_{0}),h_{i-1})
\end{eqnarray*}%
and 
\begin{equation*}
d_{F}((a_{i},h_{i-1}),h_{i})\leq d(a_{i},h_{i-1},h_{i})
\end{equation*}
whereas $d_{F}((a_{i-1},\cdots ,a_{0}),h_{i-1})$ is bounded by hypothesis.
This completes the induction step and we conclude for $i=n$ that $d_{F}(a,h)$
is bounded.

In view of the form of the axioms for a pseudometric structure on the
category $\mathbf{Free}(X,A)(x,y)$, a supremum of a family of pseudometric
structures, individually bounded on any pair of arrows, is again a
pesudometric structure. Also the supremum will satisfy the conditions for a
functorial distance. Therefore $d^{\mathrm{max}}\in \mathrm{dfunct}(X,A,d)$.
\end{proof}

We also denote by $d^{\mathrm{max}} (f,g,h)$ the value $d^{\mathrm{max}}
((f,g), (h))$.

\begin{remark}
\label{upperbound} We have the upper bound%
\begin{equation*}
d^{\mathrm{max}}(f,g,h)\leq d(f,g,h)\text{.}
\end{equation*}
\end{remark}

We would like to get a lower bound.

\begin{proposition}
\label{Clower} Suppose $F: (X,A,d)\rightarrow ({{\mathcal{C}}} ,d_{{{%
\mathcal{C}}}})$ is an AC-functor from $(X,A,d)$ to a metrized category $({{%
\mathcal{C}}} ,d_{{{\mathcal{C}}}})$. Assume Hypothesis \ref{graphcomp}.
Then we have 
\begin{equation*}
d_{{{\mathcal{C}}}}(F(g)\circ F(f),F(h)) \leq d^{\mathrm{max}} (f,g,h). 
\end{equation*}
There exists a functor on which this inequality is an equality for all $f,g,h
$.
\end{proposition}

\begin{proof}
Let $d_{F}$ be the pullback distance induced by $F$. Then $d_{F}\in \mathrm{%
dfunct}(X,A,d)$. By the construction of $d^{\mathrm{max}}$ we have $%
d_{F}\leq d^{\mathrm{max}}$, so 
\begin{eqnarray*}
d_{{{\mathcal{C}}}}(F(g)\circ F(f),F(h)) &=&d_{F}((f,g),(h)) \\
&\leq &d^{\mathrm{max}}((f,g),h)=d^{\mathrm{max}}(f,g,h).
\end{eqnarray*}%
This shows the inequality.

Let ${{\mathcal{C}}}^{\mathrm{max}}$ be the metrized category obtained from $%
\mathbf{Free}(X,A)(x,y)$ by identifying arrows $a$ and $b$ whenever $d^{%
\mathrm{max}}(a,b)=0$, as described in the paragraph at the end of Section %
\ref{metrizedcategories}. This is a metrized category with distance induced
by $d^{\mathrm{max}}$, the distance is a metric and the map $f\mapsto (f)$
defines an AC functor $F^{\mathrm{max}}$ from $(X,A,d)$ to ${{\mathcal{C}}}^{%
\mathrm{max}}$. Tautologically, the inequality is an equality for this
functor.
\end{proof}

We now apply the previous corollary to the Yoneda functors. Assume that $%
(X,A,d)$ satisfies absolute transitivity, i.e. transitivity for $\alpha =1$.
Then we have seen that the Yoneda constructions define AC functors $Y_{u}$
from $(X,A,d)$ to the metrized category $\mathbf{MetCor}$ of metric
correspondences with 
\begin{equation*}
Y_{u}(x)=(A(u,x),\varphi _{u,x})
\end{equation*}%
and for $a\in A(x,y)$, 
\begin{equation*}
Y_{u}(a)\in {\mathcal{M}}(Y_{u}(x),Y_{u}(y))\mbox{  with  }%
Y_{u}(f)(a,b):=d(a,f,b).
\end{equation*}%
If $(X,A,d)$ is absolutely transitive, then $Y_{u}$ is an AC functor
(Proposition \ref{abstransYoneda}).

\begin{corollary}
Suppose $(X,A,d)$ is absolutely transitive and satisfies Hypothesis \ref%
{graphcomp}. Then for any $u,x,y,z$ and $a\in A(u,x)$, $b\in A(u,y)$, $c\in
A(u,z)$ and $f\in A(x,y)$, $g\in A(y,z)$ and $h\in A(x,z)$ we have 
\begin{equation*}
d(a,h,c)\leq d^{\mathrm{max}}(f,g,h)+d(a,f,b) +d(b,g,c). 
\end{equation*}
For any $a,c,f,g,h$ as above, 
\begin{equation*}
\inf _{b\in A(u,y)}(d(a,f,b) + d(b,g,c)) \leq d^{\mathrm{max}}(f,g,h) +
d(a,h,c). 
\end{equation*}
\end{corollary}

\begin{proof}
For any $f\in A(x,y)$, $g\in A(y,z)$ and $h\in A(x,z)$ we have 
\begin{equation*}
\sup_{a\in A(u,x),c\in A(u,z)}\left\vert d(a,h,c)-\!\!\inf_{b\in
A(u,y)}(d(a,f,b)+d(b,g,c))\right\vert \leq d^{\mathrm{max}}(f,g,h),
\end{equation*}%
by the corollary using the fact that $Y_{u}$ is an AC-functor. We get, for
any $a$ and $c$, the two inequalities 
\begin{equation*}
d(a,h,c)-\inf_{b\in A(u,y)}(d(a,f,b)+d(b,g,c)\leq d^{\mathrm{max}}(f,g,h)
\end{equation*}%
and 
\begin{equation*}
\inf_{b\in A(u,y)}(d(a,f,b)+d(b,g,c)-d(a,h,c)\leq d^{\mathrm{max}}(f,g,h).
\end{equation*}%
The first one implies the same inequality for any $b$, giving the first
statement of the corollary. The second one gives the second statement of the
corollary.
\end{proof}

\begin{corollary}
Suppose $(X,A,d)$ is absolutely transitive and satisfies Hypothesis \ref%
{graphcomp}. Then for any $f,g,h$ we have 
\begin{equation*}
\inf _{b\in A(x,y)} \mathrm{dist} _{A(x,y)} (f,b) + d(b,g,h) \leq d^{\mathrm{%
max}}(f,g,h). 
\end{equation*}
\end{corollary}

\begin{proof}
Apply the previous corollary to $u=x$, $a=1_{x}$ and $c=h$. Then%
\begin{equation*}
d(a,c,h)=d(1_{x},h,h)=0
\end{equation*}
and%
\begin{equation*}
d(a,f,b)=d(1_{x},f,b)=\mathrm{dist}_{A(x,y)}(f,b).
\end{equation*}
\end{proof}

We obtain the following embedding theorem.

\begin{theorem}
\label{main} Suppose $(X,A,d)$ is bounded, absolutely transitive and
satisfies Hypothesis \ref{graphcomp}. Then for any $f,g,h$ we have $d^{%
\mathrm{max}}(f,g,h)=d(f,g,h)$. Hence, there exists an AC functor $F:
(X,A,d)\rightarrow ({{\mathcal{C}}} , d_{{{\mathcal{C}}} })$ to a metrized
category, such that for any $f,g,h$ we have 
\begin{equation}  \label{Fequal}
d(f,g,h)= d_{{{\mathcal{C}}}}(F(g)\circ F(f), F(h)).
\end{equation}
\end{theorem}

\begin{proof}
We have $d^{\mathrm{max}}(f,g,h)\leq d(f,g,h)$ by Remark \ref{upperbound}.
On the other hand, for any $b\in A(x,y)$ we have 
\begin{equation*}
d(f,g,h)\leq d(b,g,h) + \mathrm{dist}_{A(x,y)} (f,b), 
\end{equation*}
by Lemma \ref{phileft}. Applying the result of the previous corollary, this
gives 
\begin{equation*}
d(f,g,h)\leq d^{\mathrm{max}}(f,g,h). 
\end{equation*}
Now the last statement of Proposition \ref{Clower} gives existence of an AC
functor $F$ satisfying \eqref{Fequal}.
\end{proof}

\noindent \emph{Proof of Theorem \ref{mainintro}:} Notice that the functor
constructed in Theorem \ref{main} is the identity on the set of objects. The
target metrized category ${{\mathcal{C}}} ^{\mathrm{max}}$ is the quotient
of $\mathbf{Free}(X,A)$ by the equivalence relation induced by the
pseudo-metric $d^{\mathrm{max}}$. Assume that $(X,A,d)$ satisfies the
separation property (Definition \ref{separated}). The maps on arrow sets $%
F:A(x,y)\rightarrow {{\mathcal{C}}} (x,y)$ preserve distances, as may be
seen by taking $y=x$ and $f:= 1_x$ in the property. It follows that they are
injective, so we may consider that $A(x,y)\subset {{\mathcal{C}}} (x,y)$.
This gives the inclusion required to finish the proof of Theorem \ref%
{mainintro}. \hfill $\Box$

The following corollary may be viewed as an improvement of Theorem \ref%
{zerocat}.

\begin{corollary}
Suppose $(X,A,d)$ is bounded and $\epsilon $-categoric for all $\epsilon >0$%
. Then there exists an AC functor $F:(X,A,d)\rightarrow ({{\mathcal{C}}},d_{{%
{\mathcal{C}}}})$ to a metrized category, such that for any $f,g,h$ we have%
\begin{equation*}
d(f,g,h)=d_{{{\mathcal{C}}}}(F(g)\circ F(f),F(h))\text{.}
\end{equation*}
\end{corollary}

\begin{proof}
By Lemma \ref{epcattrans}, $(X,A,d)$ is absolutely transitive. Furthermore,
the $\epsilon $-categoric condition for any one $\epsilon $ implies
Hypothesis \ref{graphcomp}. Apply Theorem \ref{main}.
\end{proof}

\section{Paths in ${\mathbb{R}}^n$}

\label{pathsRn}

As we have noted above, the absolute transitivity condition doesn't apply to
the AC structure coming from a $2$-metric space. We look at how to calculate 
$d^{\mathrm{max}}$, but for simplicity we restrict to the standard example
of euclidean space as a $2$-metric space. In what follows, let $X:= {\mathbb{%
R}}^n$ with the standard $2$-metric: $d_X(x,y,z)$ is the area of the
triangle spanned by $x,y,z$. Let $(X,A,d)$ be the associated AC structure.
Recall that $A(x,y)=\{ \ast _{x,y}\}$.

Arrows in $\mathbf{Free}(X,A)$ have the following geometric interpretation.
An arrow from $x$ to $y$ corresponds to a composable sequence in $(X,A)$
going from $x$ to $y$, which in this case just means a sequence of points $%
(x_0,\ldots , x_k)$ with $x_0=x$ and $x_k=y$. We may picture this sequence
as being a piecewise linear path composed of the line segments $\overline{%
x_ix_{i+1}}$. Thus, we may consider elements $a\in \mathbf{Free}(X,A)(x,y)$
as being piecewise linear paths from $x$ to $y$.

Suppose $D\subset {\mathbb{R}}^2$ is a compact convex polygonal region. Its
boundary $\partial D$ is a closed piecewise linear path. If $s:D\rightarrow {%
\mathbb{R}}^n$ is a piecewise linear map, its boundary $\partial s$ is a
closed piecewise linear path in ${\mathbb{R}}^n$. We can divide $D$ up into
triangles on which $s$ is linear, and define $\mathrm{Area}(D,s)$ to be the
sum of the areas of the images of these triangles in ${\mathbb{R}}^n$. This
could also be written as 
\begin{equation*}
\mathrm{Area}(D,s)= \int _D |ds | . 
\end{equation*}

Suppose $a$ is a closed piecewise linear path in $X={\mathbb{R}}^{n}$. Put 
\begin{equation*}
\mathrm{MinArea}(a):=\inf_{\partial s=a}\mathrm{Area}(D,s)
\end{equation*}%
be the minimum of the area of piecewise linear maps from compact convex
polyhedral regions to $X$ with boundary $a$.

\noindent \emph{Remark:} $\mathrm{MinArea}(a)$ is also the minimum of areas
of piecewise $C^1$ maps from the disk, with boundary $a$.

If $f,g$ are piecewise linear paths from $x$ to $y$, let $g^{-1}f$ denote
the closed path based at $x$ obtained by following $f$ by the inverse of $g$
(the path $g$ run backwards). Define a metrized structure on the category $%
\mathbf{Free}(X,A)$ of piecewise linear paths, by 
\begin{equation*}
d_{\mathrm{Area}}(f,g):=\mathrm{MinArea}(g^{-1}f). 
\end{equation*}
The associated AC structure also denoted by $d_{\mathrm{Area}}$ is given by 
\begin{equation*}
d_{\mathrm{Area}}(f,g,h)=d_{\mathrm{Area}}(g\circ f, h) = \mathrm{MinArea}%
(h^{-1}gf). 
\end{equation*}

\begin{lemma}
Suppose two paths $f,g$ differ by an elementary move, in the sense that one
is $f=(x_0,\ldots , x_i,z,x_{i+1},\ldots , x_k)$ and the other is $%
g=(x_0,\ldots , x_i,z,x_{i+1},\ldots , x_k)$. Then 
\begin{equation*}
d^{\mathrm{max}}(f,g)\leq d_{\mathrm{Area}}(f,g)=d_X(x_i,z,x_{i+1}). 
\end{equation*}
\end{lemma}

\begin{proof}
By the axiom \eqref{mcaxiom} for a metrized category applied to the
compositions on either side with the paths $(x_0,\ldots , x_i)$ and $%
(x_{i+1}, \ldots , x_k)$, we get 
\begin{equation*}
d^{\mathrm{max}}(f,g)\leq d^{\mathrm{max}} ((x_i,z,x_{i+1}), (x_i,x_{i+1})). 
\end{equation*}
By the definition of $d^{\max}$, 
\begin{equation*}
d^{\mathrm{max}} ((x_i,z,x_{i+1}), (x_i,x_{i+1})) \leq d(\ast _{x_i,z}, \ast
_{z,x_{i+1}}, \ast _{x_i,x_{i+1}}) = d_X(x_i,z,x_{i+1}). 
\end{equation*}
Note that the minimal area of a disk whose boundary is a triangle, is the
area of the triangle, so $d_X(x_i,z,x_{i+1})$ is also equal to $d_{\mathrm{%
Area}}(f,g)$.
\end{proof}

\begin{theorem}
For $(X,A)$ the AC structure associated to the $2$-metric space $X={\mathbb{R%
}}^{n}$ with its standard area metric, on the category of piecewise linear
paths $\mathbf{Free}(X,A)$ we have $d^{\mathrm{max}}=d_{\mathrm{Area}}$. In
particular, for $f\in A(x,y)$, $g\in A(y,z)$ and $h\in A(x,z)$ we have%
\begin{equation*}
d(f,g,h)=d_{X}(x,y,z)=d^{\mathrm{max}}(g\circ f,h)\text{.}
\end{equation*}
\end{theorem}

\emph{Sketch of proof:} The pseudo-metric $d_{\mathrm{Area}}$ gives a
structure of pseudo-metrized category on $\mathbf{Free}(X,A)$, which agrees
with $d$ on $(X,A)$. It follows from the definition of $d^{\mathrm{max}}$
that $d^{\mathrm{max}} \geq d_{\mathrm{Area}}$.

In the other direction, suppose $f,g\in \mathbf{Free}(X,A)(x,y)$ and suppose
we are given a piecewise linear map $s$ from a polyhedron $D$ to ${\mathbb{R}%
}^{n}$ with boundary $\partial s=g^{-1}f$. Dividing the polyhedron up into
triangles, we obtain a sequence of paths $f_{0}=f,f_{1},\ldots ,f_{m}=g$
such that $f_{i}$ and $f_{i+1}$ differ by an elementary move as in the
previous lemma, and the triangles occurring in these elementary moves
together make up the polyhedron $D$. Applying the previous lemma, and in
view of the fact that $\mathrm{Area}(D,s)$ is the sum of the areas of the
triangles in ${\mathbb{R}}^{n}$ in the image of $D$, we have 
\begin{equation*}
\sum_{i=1^{k}}d^{\mathrm{max}}(f_{i-1},f_{i})\leq \mathrm{Area}(D,s).
\end{equation*}%
By the triangle inequality for $d^{\mathrm{max}}$ we get 
\begin{equation*}
d^{\mathrm{max}}(f,g)\leq \mathrm{Area}(D,s),
\end{equation*}%
and since $d_{\mathrm{Area}}(f,g)$ is the minimum of $\mathrm{Area}(D,s)$
over all choices of $D,s$, this shows that $d^{\mathrm{max}}\leq d_{\mathrm{%
Area}}$. Therefore $d^{\mathrm{max}}=d_{\mathrm{Area}}$. \hfill $\Box $

\section{Further questions}

\label{questions}

\subsection{More transitivity conditions}

\label{moretrans} Various other transitivity conditions may be considered.
For example, a weaker \emph{$\leq \epsilon$ transitivity axiom} would
require existence of $m$ only when the right hand side is $\leq \epsilon$.

The \emph{small transitivity condition} is that if $f\in A(x,y)$ and $g\in
A(y,z)$ then there exists $m\in A(x,z)$ such that 
\begin{equation*}
d(f,g,m)\leq \mathrm{min}\{ \alpha (f),\alpha (g)\} . 
\end{equation*}

We could also weaken the basic transitivity conditions of Definition \ref%
{abstrans}, for example by allowing an error of some $\epsilon$. It would be
interesting to see what kinds of lower bounds for $d^{\mathrm{max}}$ could
be obtained with these conditions.

\subsection{Lower bound for $\protect\alpha$-transitive AC structures}

\label{alphalb} In Theorem \ref{main} we used the absolute transitivity
condition. However, in examples such as the AC-structure coming from a
transitive $2$-metric space, only transitivity relative to an amplitude $%
\alpha$ holds. Therefore, it is an interesting question to see what kind of
lower bound for $d^{\mathrm{max}}$ could be obtained using $\alpha$%
-transitivity.

\subsection{Paths in a $2$-metric space}

\label{plpath} In the example of a $2$-metric space $(X,d_{X})$, we put 
\begin{equation*}
A(x,y):=\{\ast _{x,y}\}\text{ and }d(\ast _{x,y},\ast _{y,z},\ast
_{x,z}):=d_{X}(x,y,z).
\end{equation*}
We can then construct the induced pseudometric $d^{\mathrm{maxs}}$ on $%
\mathbf{Free}(X,A)$. Morphisms in $\mathbf{Free}(X,A)$ may be viewed as
paths in $X$. If $X={\mathbb{R}}^{n}$ then we view such a path as a
piecewise linear path in ${\mathbb{R}}^{n}$, replacing $\ast _{x,y}$ by the
straight line segment joining $x$ to $y$, and we have seen in Section \ref%
{pathsRn} that $d^{\mathrm{max}}(f,g)$ is the minimal area of a disk whose
boundary consists of the paths $f$ and $g$, as in the example of \ref%
{expaths}. It will be interesting to generalize this to arbitrary $2$-metric
spaces.

\subsection{Natural transformations}

\label{nattrans}

Classically the next step after functors is to consider natural
transformations. It is an interesting question to understand the appropriate
generalization of this notion to the approximate categorical context.

Suppose $F,G:(X,A)\rightarrow (Y,B)$ are two prefunctorial maps. A \emph{%
pre-natural transformation} $\eta :F\rightarrow G$ is a function which for
any $x,y\in X$ and $f\in A(x,y)$ associates $\eta (f)\in B(Fx,Gy)$. We say
that $\eta $ is a \emph{$k$-natural transformation} if, for any $x,y,z\in X$
and $f\in A(x,y)$, $g\in A(y,z)$ and $h\in A(x,z)$ we have 
\begin{equation*}
d(F(f),\eta (g),\eta (h))\leq kd(f,g,h)\mbox{  and  }d(\eta (f),G(g),\eta
(h))\leq kd(f,g,h).
\end{equation*}

If $F$ is $k$-functorial then defining $\eta (f):= F(f)$ gives an ``identity
prenatural transformation'' from $F$ to itself, and it is also $k$-natural.

Suppose $F,G,H: (X,A)\rightarrow (Y,B)$ are three prefunctorial maps.
Suppose $\eta : F\rightarrow G$, $\zeta : G\rightarrow H$ and $\omega :
F\rightarrow H$ are prenatural transformations.

For any $x,y,z\in X$ and $f\in A(x,y)$, $g\in A(y,z)$ and $h\in A(x,z)$,
consider 
\begin{equation*}
Fx \overset{\eta (f)}{\rightarrow} Gy \overset{\zeta (g)}{\rightarrow} Fz 
\end{equation*}
and ask that 
\begin{equation*}
d(\eta (f), \zeta (g), \omega (h)) \leq kd(f,g,h) + \delta_0 (\eta , \zeta ,
\omega ). 
\end{equation*}
Let $\delta (\eta , \zeta , \omega )$ be the $\mathrm{inf}$ of the $\delta
_0(\eta , \zeta , \omega )$ which work here. We hope that this will allow to
define an approximate categorical structure on the functors and natural
transformations. This is left open as a question for the future.

\subsection{Correspondence functors}

\label{cfunctors}

It was useful to introduce a notion of metric correspondence between metric
spaces. One may ask whether this notion can be extended naturally to
metrized categories and AC structures.

Let $(X,A,d_{A})$ and $(Y,B,d_{B})$ be AC structures. We would like to
define a notion of functor from $A$ to $B$ using the idea of metric
correspondences on the morphism sets. Let us try as follows. Consider a
function $F:X\rightarrow Y$ and functions $f(x,x^{\prime },a,b)\in {\mathbb{R%
}}$ for any $x,x^{\prime }\in X$ and $a\in A(x,x^{\prime })$ and $b\in
B(Fx,Fx^{\prime })$. Roughly speaking we would like to have 
$$
\inf_{a^{\prime \prime }}(f(x,x^{\prime \prime },a^{\prime \prime
},b^{\prime \prime })+d_{A}(a,a^{\prime },a^{\prime \prime }) 
$$
$$
\leq \inf_{b,b^{\prime }}f(x,x^{\prime },a,b)+f(x^{\prime },x^{\prime \prime
},a^{\prime },b^{\prime }) 
+d_{B}(b,b^{\prime },b^{\prime \prime }).
$$
This translates also into: for any $x,x^{\prime },x^{\prime \prime }\in X$,
any $a\in A(x,x^{\prime })$, any $a^{\prime }\in A(x^{\prime },x^{\prime
\prime })$ and any $a^{\prime \prime }\in A(x,x^{\prime \prime })$, and any $%
b^{\prime \prime }\in B(Fx,Fx^{\prime \prime })$ we have 
$$
\inf_{
\begin{array}{c}
{\scriptstyle b\in B(Fx,Fx^{\prime }) }
\\
{\scriptstyle b^{\prime }\in B(Fx^{\prime }, Fx^{\prime
\prime })}
\end{array}}
\left(
\begin{array}{c}
f(x,x^{\prime },a,b)+f(x^{\prime },x^{\prime \prime },a^{\prime
},b^{\prime })
-f(x,x^{\prime \prime },a^{\prime \prime },b^{\prime \prime })
\medskip
\\
+d_{B}(b,b^{\prime },b^{\prime \prime }) 
\end{array}
\right)
$$
$$
\leq  d_{A}(a,a^{\prime
},a^{\prime \prime }).
$$

\subsection{Reconstruction questions}

\label{reconstruction}

Given a category $(X,C, \circ )$ with a metric structure $\psi$, we get an
approximate categorical structure by Proposition \ref{mcac}.

We could then consider a collection of subsets $A(x,y)\subset C(x,y)$ such
that $A(x,x)$ contains $1_x$. This will again give an AC structure.

One question will be, to what extent can we recover the structure of $C$
just by knowing $(X,A,d)$? Our main theorem provides an existence result in
the absolutely transitive case. One may then ask, to what extent is the
enveloping metrized category unique?

Another somewhat similar question: suppose $A(x,y)=C(x,y)$. However, suppose 
$d^{\prime }$ is a different AC-structure obtained by perturbing the original one.
For example, note that $d^{\prime }=d+\epsilon d_1$ is again an AC-structure for any
AC-structure $d_1$. Question: can we recover the categorical structure, i.e. the
composition $\circ$, from the perturbed $d^{\prime }$?

The AC-structures on $(X,A)$ form a cone, because if $d_1$ and $d_2$ are
AC-structures and $c_1,c_2$ are positive constants then $c_1d_1 + c_2d_2$ is
again an AC-structure. So, another question is, what properties does this
cone have? What do the boundary points or extremal rays look like?

The picture in subsection \ref{exfinite} suggests that we should look for the 
position of the categorical structures within this cone. 
To recast the questions of two paragraphs ago,
given a finite graph, are there small values of $\epsilon$ such
that any $\epsilon$-categorical structure is near to a $0$-categorical structure? 

\subsection{Category theory}

\label{cattheory}

To what extent can we generalize the classical structures and constructions
of category theory to the situation of metrized categories and AC structures?

\subsection{Fixed points and optimization}

\label{fixedpoints}

One of our main original motivations for looking at $2$-metric spaces in 
\cite{AS1, AS2} was to consider the theory of fixed points and other fixed
subsets such as lines. In the more general setting of AC structures, we can
envision several different kinds of fixed point and iteration problems, for
example fixed points of a functor, fixed points of a metric correspondence
between metric spaces, as well as fixed arrows within a metrized category or
AC structure. It will be interesting to see what applications these might
have to optimization problems.

\end{document}